\documentclass{amsart}
\newtheorem{theorem}{Theorem}[section]
\newtheorem{lemma}[theorem]{Lemma}
\newtheorem{corollary}[theorem]{Corollary}
\newtheorem{proposition}[theorem]{Proposition}

\theoremstyle{definition}
\newtheorem{definition}[theorem]{Definition}

\theoremstyle{remark}
\newtheorem{remark}[theorem]{Remark}
\numberwithin{equation}{section}

\newcommand{\eps}{\varepsilon}
\newcommand{\Le}{\mathcal{L_{\eps}}}
\newcommand{\Lw}{\mathcal{L_{\omega}}}
\newcommand{\LT}{\mathcal{L}_0}
\newcommand{\LTl}{\mathcal{L}_{l,0}}

\newcommand{\LH}{\hat{\mathcal{L}}_{l,\eps}}
\newcommand{\Lh}{\hat{\mathcal{L}}_{\eps}}

\usepackage{amsfonts}
\usepackage{color,graphicx,amssymb}
\usepackage{graphicx}
\begin{document}
\title{Escape Rates Formulae and Metastability for Randomly perturbed maps}
\author{Wael Bahsoun}
    \address{Department of Mathematical Sciences, Loughborough University,
Loughborough, Leicestershire, LE11 3TU, UK}
\email{W.Bahsoun@lboro.ac.uk}
\author{Sandro Vaienti}
\address
{UMR-7332 Centre de Physique Th\'{e}orique, CNRS, Universit\'{e}
d'Aix-Marseille, Universit\'{e} du Sud, Toulon-Var and FRUMAM,
F\'{e}d\'{e}ration de Recherche des Unit\'{e}s des Math\'{e}matiques de Marseille,
CPT Luminy, Case 907, F-13288 Marseille CEDEX 9}
\email{vaienti@cpt.univ-mrs.fr}
\thanks{We would like to thank Anthony Quas for useful comments. We would also like to thank J\"org Schmeling for useful discussions on topics addressed in this work. W.B. would like to thank LMS for supporting a visit of S.V. to the Department of Mathematical Sciences at Loughborough University where this work was initiated. S.V. research was supported by the ANR ``Perturbations", and by the CNRS-PEPS  ``Mathematical methods of climate models" and by the CNRS-PICS N. 05968. Part of this work was done while S.V.  was visiting the {\em Centro de Modelamiento Matem\'{a}tico, UMI2807}, in Santiago de Chile with a CNRS support (d\'el\'egation)}
\begin{abstract}
We provide escape rates formulae for piecewise expanding interval maps with `random holes'. Then we obtain rigorous approximations of invariant densities of randomly perturbed metabstable interval maps. We show that our escape rates formulae can be used to approximate limits of invariant densities of randomly perturbed metastable systems.
\end{abstract}
\maketitle
\pagestyle{myheadings}
\markboth{Escape Rates Formulae and Metastablilty}{W. Bahsoun And S. Vaienti}
\section{Introduction}
A dynamical system is called open if there is a subset in the phase space, called a \textit{hole}, such that whenever an orbit lands in it, the dynamics of this obit is terminated. In open dynamical systems, long-term
statistics are described by a conditionally invariant measure and its related escape rate, measuring the mass lost from the system per unit time \cite{DY,Det}. For the past three years there has been a considerable interest in describing the escape rate of an open system as a function of the hole's position and size. In \cite{BY} it was observed that for the doubling map, given two holes, $H_1$ and $H_2$ of the same fixed size, each centred around a periodic point, say $x_1, x_2$ respectively, with the period of $x_1$ smaller than that of $x_2$, then the escape rate through $H_1$ is bigger than that through $H_2$. Later in \cite{KL2} it was shown that, by shrinking the hole to a point, say $x_0$, the escape rate depends on two things: i) whether $x_0$ is periodic or not; ii) the invariant density of the corresponding closed system. Following the success of \cite{BY, KL2}, other researchers studied this phenomenon for different types of systems \cite{AB, DW1, FP1}\footnote{\cite{FP1, DW1} are in the direction of \cite{KL2}, while \cite{AB} is in the spirit of \cite{BY}.}. These results have lead to insights in studying metastable dynamical systems which behave approximately like a collection of open systems: the infrequent transitions between almost invariant regions in a metastable system are similar to infrequent escapes from associated open systems \cite{GHW, DW, FMS, KL2}.

\bigskip

This research topic is currently very active in ergodic theory and dynamical systems. In part, this is due to the interesting applications of open and metastable dynamical systems in physical sciences. For instance, open dynamical systems are used to study transport in heat conduction \cite{GG}. They also play an important in astronomy \cite{WBW}. In molecular dynamics, almost invariant regions of metastable systems are used to identify sets where stable molecular conformations occur \cite{Molec}. In astrodynamics, metastable systems are used to recognize regions from which asteroids escape is infrequent \cite{Astro}. Moreover, metastable dynamical systems have been recently used to develop realistic models of atmospheric and ocean circulation \cite{DF, FP, SFM}. In these models, it is believed that metastable states\footnote{Metastable states correspond to eigenvalues of a transfer operator which are very close to 1, a phenomenon which is very transparent in the metastable systems of this paper.} lie behind long term global circulation patterns, and form large scale barriers to transport.

\bigskip

In this paper we first study piecewise smooth and expanding interval maps with `random holes'. We prove that the random open system admits an absolutely continuous conditionally stationary measure (accsm). Moreover we obtain escape rates formulae (first order approximations) depending on the position of the holes. An older escape rate formula for randomly perturbed maps with specific holes positions and a specific distribution on the noise space was obtained in \cite{CMS}. Our result is in the spirt of the recent trend of describing the escape rate as a function of the position of the hole \cite{AB, BY, FP1, KL2}. In particular, our result generalizes Keller-Liverani escape rates formulae \cite{KL2} which were obtained for deterministic perturbations.

\bigskip

In the second part of the paper we study random
perturbations of interval maps that initially admit
 exactly two invariant ergodic densities. However,
 under random perturbations that generate \textit{random holes}
  and allow leakage of mass between the two initially ergodic subsystems,
  the random system admits a unique stationary  density.
  We show that such a  density for the metastable random system
  can be approximated in the $L^1$ norm (with respect to the ambient
   Lebesgue measure $m$),  by a particular convex combination of
   the two invariant ergodic densities of the initial system. In particular,
    we show that the ratio of the weights in the convex combination is
    equal to the ratio of the averages of the measures of the left and
    right random holes. Moreover, as a by product of our escape rates
    formulae, we show that these weights can be also identified as the
    ratio of the escape rates from the left and right random open systems.

 \bigskip

We finish the introduction by addressing possible generalizations of our work. We foresee two natural generalizations. The first generalization concerns the extension to piecewise expanding maps in higher dimensions. Two difficulties arise in this setting. First, obtaining a Lasota-Yorke inequality for the operator associated with the open system with random holes, see Lemma \ref{LeU} would require a more careful control of the variation, whatever the Banach space ${\bold B}$ that operator acts on is, of the characteristic function along the boundaries of the partitions, which are now codimension-1 piecewise smooth submanifolds and not points as in the setting of this paper. The second difficulty is to get local smoothness of elements of ${\bold B}$ around the holes similar to Lemma 1 of \cite{GHW}. We believe that the latter property is fundamental to obtain results on approximating invariant densities of metastable systems, even in the one-dimensional case.

\bigskip

The second generalization is the extension to non-uniformly expanding maps. Metastability for deterministic maps with indifferent fixed points was treated by us in \cite{BV}. One of the main tools we used there was to {\em reduce} the study of metastability to an induced subsystem where the map is uniformly expanding, and then pullback the result to the original system. It is not clear how this technique can be implemented in the presence of noise. A few recent results \cite{AV, Shen} used induction to achieve stochastic stability in some classes of one-dimensional maps. It would be interesting to explore the possibility of apply those techniques to metastability. Supposing the analogous of our Theorem \ref{main} can be proved in the random case, it would be also interesting to explore if something similar to Corollary \ref{Co1} (see section \ref{Meta}) would still hold. In this paper, the latter is related to the \textit{exponential} escape rate which is in turn associated to the existence of a spectral gap of the transfer operator. This is a `luxury' that is not available in a non-uniformly hyperbolic setting.

\bigskip

In section \ref{Open} we obtain escape rates formulae for expanding interval maps with random holes. Our main result in this section is Theorem \ref{main1}. In section \ref{Meta} we study rigorous approximations in the $L^1$-norm of invariant densities of randomly perturbed metastable maps. Our main result in this section is Theorem \ref{main}. 
\section{Escape Rate Formulae For Randomly Perturbed Maps}\label{Open}
\subsection{Notation}
Throughout the paper we use the following notation: $(I, \mathfrak B, m)$ is the measure space where $I = [0, 1]$, $\mathfrak B$ is the Borel $\sigma$-algebra and
$m$ is Lebesgue measure. For $f\in L^1(I, \mathfrak B, m)$, we define
$$V f = \inf\{\text{var}\bar f:\, f = \bar f \text{ a.e.}\},$$
where
$$\text{var}\bar f =\sup\{\sum_{j=0}^{t-1}|f(x_{j+1})-f(x_j)|:\,0=x_0 <x_1 <\dots<x_{t} =1\}.$$
We denote by $BV(I)$ the space of functions of bounded variation on $I$ equipped with the norm $||\cdot||_{BV} = V(\cdot)+||\cdot||_1$, where $||\cdot||_1$ is the $L^1$ norm with respect to $m$.
\subsection{Piecewise expanding maps}
Let $T: I\to I$ be a map which satisfies the following conditions:\\
\noindent {\bf (O1)} There exists a partition of $I$, which consists of intervals $\{I_i\}_{i=1}^{q}$, $I_i\cap I_j=\emptyset$ for $ i\not= j$, $\bar I_i :=[c_{i,0},c_{i+1,0}]$ and there exists $\delta>0$  such that $T_{i,0}:= T|_{(c_{i,0},c_{i+1,0})}$ is $C^2$ which extends to a $C^2$ function $\bar T_{i,0}$ on a neighbourhood $[c_{i,0}-\delta,c_{i+1,0}+\delta]$ of $\bar I_i $ ;\\
\noindent {\bf (O2)} $\inf_{x\in I\setminus\mathcal C_0}|T'(x)|\ge\kappa^{-1}>1$, where $\mathcal C_0=\{c_{i,0}\}_{i=1}^{q}$.\\
\noindent {\bf (O3)} $T$ preserves a unique acim\footnote{The existence of $\mu$ follows from the well known result of \cite{LY}.} $\mu$ which is equivalent to $m$. Moreover, the system $(I,\mu, T)$ is mixing.
\subsection{Random Holes}
Let $(\omega_k)_{k\in\mathbb N}$ be an i.i.d.  stochastic process with values in the interval $\Omega_{\eps} =[0,\eps]$, $\eps>0$, and with probability distribution $\theta_{\eps}$. We  fix $z\in (0,1)$ and we  associate with each $\omega\in\Omega_{\eps}$ an interval $H_{\omega}$ such that $z\in H_{\omega}\subseteq H_{\eps}$, and  we assume that  $H_{\eps}\subseteq H_{\eps'}$ for $\eps\le\eps'$. Further, \\
\noindent {\bf (O4)}  We assume that $T$
is continuous at $z$; this assumption will be explicitly used in the proof of Theorem \ref{main1}. In order to apply the results of this section to section \ref{Meta}, we will also assume that the density $\rho$ is continuous at $z$. \\
\begin{figure}[htbp] 
   \centering
   \includegraphics[width=2.5in]{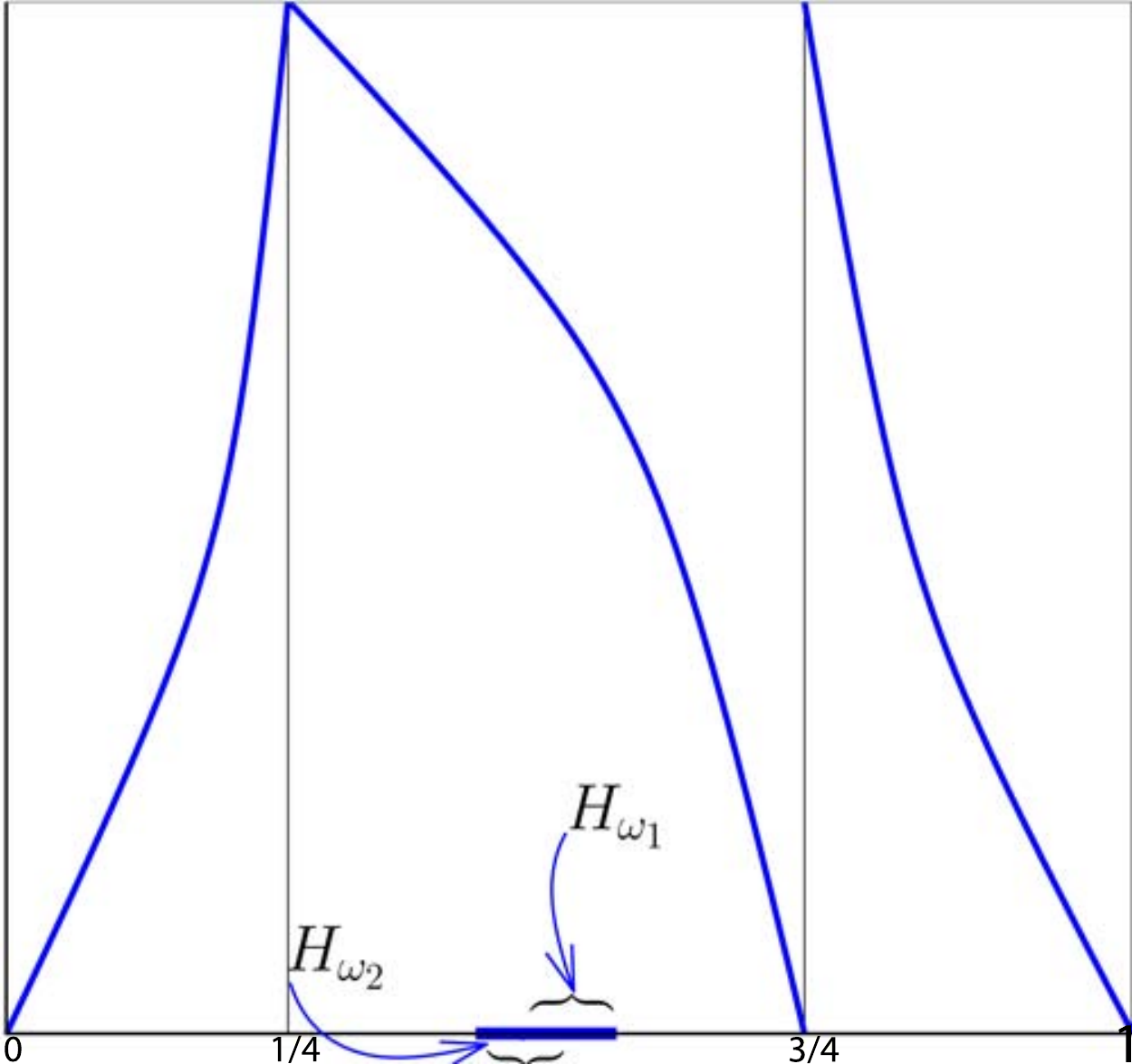}
   \caption{A map with random holes $H_{\omega_1}$ and $H_{\omega_2}$ sitting inside $H_{\eps}$. Note that $H_{\omega_1}$ and $H_{\omega_2}$ are not nested.}
\end{figure}
\subsection{Transfer operator of the random open system}
Our goal in this section is to study the existence of absolutely continuous conditionally stationary measures (accsm) and their associated escape rates through the random holes $H_{\omega}$ around a given point $z$ under the dynamics of $T$. An example of an interval map with random holes is shown in Figure 1. We set
$$X_{\omega}:=I\setminus H_{\omega},$$
and define for $f\in L^1(I,\mathfrak{B}, m)$
\begin{equation}\label{OPT}
\Lh f(x):=\int_{\Omega_\eps}\LT (f\mathbf{1}_{X_{\omega}})(x)d\theta_\eps,
\end{equation}
where $\LT$ is the transfer operator (Perron-Frobenius) \cite{Ba, BG} associated with $T$; i.e., for $f\in L^{1}(I,\mathfrak B, m)$ we have
$$\LT f(x)=\sum_{y:=T^{-1}(x)}\frac{f(y)}{|T'(y)|}.$$
The transfer operator $\Lh$ will be used to prove  that the random open system admits an accsm with exponential escape rate. By using the fact that for any measurable set $A$ and integrable function $f$
\begin{equation*}
{\bf 1}_A\LT f=\LT({\bf 1}_{T^{-1}A} f),
\end{equation*}
we obtain
\begin{equation}\label{it-hole}
\begin{split}
\Lh^nf(x)&=\int_{\Omega_{\eps}}\cdots\int_{\Omega_{\eps}}\LT^n(f{\bf 1}_{X_{\omega_1}\cap T^{-1}X_{\omega_2}\cap\cdots\cap T^{-(n-1)}X_{\omega_n}})(x)d\theta_{\eps}(\omega_1)\cdots d\theta_{\eps}(\omega_n)\\
&\overset{\text{def}}{:=}\int_{\bar\Omega_{\eps}}\LT^n(f{\bf 1}_{X_{\omega_1}\cap T^{-1}X_{\omega_2}\cap\cdots\cap T^{-(n-1)}X_{\omega_n}})(x)d\theta_{\eps}^{\infty}(\bar\omega).
\end{split}
\end{equation}
where $\theta_{\eps}^{\infty}:=\Pi_{i=1}^{\infty}\theta_{\eps}$ and $\bar\omega:=(\omega_1, \omega_2,\dots)$.
\subsection{Statement of the main result of section \ref{Open}}
The following theorem provides random versions of Keller-Liverani escape rates formulae \cite{KL2}.
\begin{theorem}\label{main1}
For sufficiently small $\eps>0$, there exists an $e_{\eps}$, $0<e_{\eps}<1$, and a $g\in BV(I)$, $g_{\eps}(x)> 0$, with $\int_{I}g_{\eps,}dm=1$,  such that
$$\hat{\mathcal{L}}_{\eps}g_{\eps}=e_{\eps}g_{\eps}.$$
Moreover,
\begin{enumerate}
\item If $z$ is a non-periodic point of $T$, then
$$\lim_{\eps\to 0}\frac{1-e_{\eps}}{\int_{\Omega_\eps}\mu(H_{\omega})d\theta_{\eps}(\omega)}=1.$$
\item If $z$ is a periodic point of $T$ of minimal period $p$, $T^p$ is $C^1$ in a neighbourhood of $z$, and the probability distribution $\theta_\eps$ on the noise space $\Omega_{\eps}$ satisfies the following  condition: $\exists$ $\upsilon>1$ such that\\

{\bf (C)} \ $\theta_{\epsilon}\{\omega;\ m(H_{\omega})\in (\epsilon-\epsilon^{\upsilon}, \epsilon)\}>1-\epsilon^{\upsilon};$\\
then
$$\lim_{\eps\to 0}\frac{1-e_{\eps}}{\int_{\Omega_\eps}\mu(H_{\omega})d\theta_{\eps}(\omega)}=1-\frac{1}{|(T^p)'(z)|}.$$
\end{enumerate}
\end{theorem}
\begin{remark}\label{Re2}
\text{}
\begin{enumerate}
\item In (2) of Theorem \ref{main1}, condition {\bf (C)} means that for a fixed $\eps>0$ most of the random holes $H_{\omega}$ are not too small when compared to the size of $H_{\eps}$. Notice that the two ``${\upsilon}$" in the condition can be different, provided they remain larger than $1$. We considered the same $\upsilon$ for simplicity \footnote{It is very easy to construct examples that satisfy condition {\bf (C)}. For instance, suppose we associate to any $\omega \in \Omega_{\eps}$ a symmetric hole around the point $z$. Then: (i) $H_{\eps}=[z-\frac{\eps}{2};\ z+\frac{\eps}{2}]$; $H_{\omega}=[z-\frac{\omega}{2};\ z+\frac{\omega}{2}]$; (ii) $m(H_{\eps})=\eps;\  m(H_{\omega})=\omega$. Let us choose an absolutely continuous $\theta_{\eps}$ with density $d_{\eps}$. Our condition {\bf (C)} will be satisfied whenever $\int_{\eps-\eps^{\upsilon}}^{\eps}d_{\eps}(\omega)d\omega>1-\eps^{\upsilon}$ which can be obtained, for instance, by taking $d_{\eps}=\eps^{-\upsilon}$ on the interval $(\eps-\eps^{\upsilon}, \eps)$ and $0$ everywhere else.}. Condition {\bf (C)} is needed to insure a lower bound on the escape rate in case (2) of Theorem \ref{main1}.
\item For the existence of the accsm we required $\eps$ to be sufficiently small since we use the perturbation result of \cite{KL1} in the proof our Lemma \ref{Le4}. The strict positivity of the density $g_{\eps}$ also follows from \cite{KL1} since $\mu$ is equivalent to $m$.
\end{enumerate}
\end{remark}
\begin{remark}\label{Re3}
Since $\Lh g_{\eps}=e_{\eps}g_{\eps}$, with $\int_I g_{\eps}dm=1$, by using the definition of $\Lh$, it follows that:
$$
e_{\eps}=\int_{\Omega_{\eps}} d\theta_{\eps}(\omega)\int_I g_{\eps}\bold{1}_{X_{\omega}}dm.
$$

\noindent Set $\nu_{\eps}$ to be the Borel probability measure:
$$
\nu_{\eps}(A)\overset{\text{def}}{:=}\frac{1}{e_{\eps}}\int_I \int_{\Omega_{\eps}}  \bold{1}_A \bold{1}_{X_{\omega}}g_{\eps}d\theta_{\eps}dm,
$$
where $A\subseteq I$ is a measurable set. By using the fact that $\Lh g_{\eps}=e_{\eps}g_{\eps}$ and the definition of $\Lh$, we obtain that $\nu_{\eps}$ satisfies the following:
\begin{equation}
\nu_{\eps}(A)=\frac{1}{e_{\eps}}\int_{\Omega_{\eps}} d\theta_{\epsilon}(\omega)\nu_{\epsilon}(T^{-1}A\cap T^{-1}X_{\omega});
\end{equation}
\begin{equation}
e_{\eps}=\int_{\Omega_{\eps}} d\theta_{\eps}(\omega)\nu_{\eps}(T^{-1}X_{\omega});
\end{equation}
and
\begin{equation}
e_{\eps}^n\nu_{\eps}(A)=\int_{\bar\Omega_{\eps}}  d\theta^{\infty}_{\eps}(\bar{\omega})\nu_{\eps}(T^{-n}A\cap T^{-1}X_{\omega_1}\cap T^{-2}X_{\omega_2}\cap\cdots\cap T^{-n}X_{\omega_n}).
\end{equation}
\end{remark}
\begin{definition}
We will call $-\ln e_{\epsilon}$ the {\em escape rate for the random system}, and  $\nu_{\epsilon}$ the {\em absolutely continuous conditionally  stationary measure}.
\end{definition}
\begin{remark}
It is interesting to remark that the items (1) and (2) of Theorem \ref{main1} provide explicit perturbation formulae for the leading eigenvalue of the operator $\Lh$. This operator enjoys all the properties of what Keller \cite{Keller} recently called {\em rare events Perron-Frobenius operators}.
\end{remark}
\subsection{Proof of Theorem \ref{main1}}
We start this subsection by proving a uniform Lasota-Yorke inequality for $\Lh$ and $\LT$. A deterministic version of Lemma \ref{LeU} can be found in \cite{LM}.
\begin{lemma}\label{LeU}
There exists a $\gamma\in (0,1)$ and constants $A, B> 0$ such that for any $n\ge 1$ and $f\in BV(I)$ we have
$$||\LT^nf||_{BV}\le A\gamma^n||f||_{BV}+B||f||_1;$$
$$||\Lh^nf||_{BV}\le A\gamma^n||f||_{BV}+B||f||_1.$$
 \end{lemma}
\begin{proof} We begin to compute the total variation, leaving the $L^1$ estimate at the end. Let $\mathcal Z^n=\mathcal Z\vee T^{-1}\mathcal Z\dots\vee T^{-(n-1)} \mathcal Z$, where $\mathcal Z=\{I_i\}_{i=1}^{q}$. Let $J_n=\frac{1}{|(T^n)'|}$. For $A\in \mathcal Z^n$ and $f\in BV(I)$:
there is a $\kappa\in (0,1)$, such that
\begin{equation}\label{LY-hole0}
V{\bf 1}_{T^n A}(J_n f)\circ T^{-n}_{|A}\le 2\kappa^nV_{A} f+\frac{1}{\min_{A\in\mathcal Z^n}m(A)}\int_{A}|f|dm.
\end{equation}
In particular,
\begin{equation}\label{PFU}
V_{I}\LT^n(f)\le \sum_{A\in \mathcal Z^n}V{\bf 1}_{T^n A}(J_n f)\circ T^{-n}_{|A}\le 2\kappa^nV_{I} f+\frac{1}{\min_{A\in\mathcal Z^n}m(A)}\int_{I}|f|dm.
\end{equation}
Now, let us  consider a fixed random path of length $n$, $(\omega_1,\omega_2,\dots,\omega_n)$, $f\in BV(I)$ and define
$$g:=f{\bf 1}_{X_{\omega_1}\cap T^{-1}X_{\omega_2}\cap\cdots\cap T^{-(n-1)}X_{\omega_n}}$$
and observe that $g\in BV(I)$. Then using (\ref{LY-hole0}) and noticing that
$$\left(X_{\omega_1}\cap T^{-1}X_{\omega_2}\cap\cdots\cap T^{-(n-1)}X_{\omega_n}\right)\cap A$$
consists of at most $n+1$ connected components, we obtain
\begin{equation} \label{LY-hole1}
\begin{split}
& V{\bf 1}_{T^n A}(J_n f{\bf 1}_{X_{\omega_1}\cap T^{-1}X_{\omega_2}\cap\cdots\cap T^{-(n-1)}X_{\omega_n}})\circ T^{-n}_{|A}\\
&=V{\bf 1}_{T^n A}(J_n g)\circ T^{-n}_{|A} \le 2\kappa^nV_{A} g+\frac{1}{\min_{A\in\mathcal Z^n}m(A)}\int_{A}|g|dm\\
&\le 2\kappa^n(V_{A}f+2(n+1)\sup_{x\in A}|f|)+\frac{1}{\min_{A\in\mathcal Z^n}m(A)}\int_{A}|f|dm\\
&\le  2\kappa^n(2n+3)V_{A}f+\frac{1}{\min_{A\in\mathcal Z^n}m(A)}(2\kappa^n(2n+2)+1)\int_{A}|f|dm.
\end{split}
\end{equation}
Summing over $A\in\mathcal Z^n$ in  (\ref{LY-hole1}) we obtain
\begin{equation} \label{LY-hole2}
\begin{split}
&V_{I}\LT^n(f{\bf 1}_{X_{\omega_1}\cap T^{-1}X_{\omega_2}\cap\cdots\cap T^{-(n-1)}X_{\omega_n}})\\
&\hskip 2.5cm \le  2\kappa^n(2n+3)V_{I}f+\frac{1}{\min_{A\in\mathcal Z^n}m(A)}(2\kappa^n(2n+2)+1)\int_{I}|f|dm
\end{split}
\end{equation}
Since the inequality in (\ref{LY-hole2}) does not depend on $\omega$ and
\begin{equation*}
\Lh^nf(x)=\int_{\Omega_{\eps}}\cdots\int_{\Omega_{\eps}}\LT^n(f{\bf 1}_{X_{\omega_1}\cap T^{-1}X_{\omega_2}\cap\cdots\cap T^{-(n-1)}X_{\omega_n}})(x)d\theta_{\eps}(\omega_1)\cdots d\theta_{\eps}(\omega_n),
\end{equation*}
we also have
\begin{equation} \label{LY-hole3}
V_{I}\Lh^nf\le 2\kappa^n(2n+3)V_{I}f+\frac{1}{\min_{A\in\mathcal Z^n}m(A)}(2\kappa^n(2n+2)+1)\int_{I}|f|dm.
\end{equation}
 The estimate on $||\Lh^nf||_1$ is easy. Indeed, this can be done by splitting $f$ into its positive and negative parts and by using the linearity of the transfer operator.
 Therefore, we may suppose that $f$ is non negative.
  This allows us to interchange the integrals w.r.t. the Lebesgue measure and $\theta_{\eps}^{\infty}$ and to use duality. In conclusion we get
$$
||\Lh^nf||_1\le \int |f|{\bf 1}_{X_{\omega_1}\cap T^{-1}X_{\omega_2}\cap\cdots\cap T^{-(n-1)}X_{\omega_n}})dm\le ||f||_1.
$$
Since there exists $n_0$ and $\gamma\in (\kappa,1)$ such that $2\kappa^{n_0}(2n_0+3)\le \gamma^{n_0}$,
we can choose $A:=2n_0+3$, $B:=\sup_{n\le n_0}\frac{1}{\min_{A\in\mathcal Z^n}m(A)}2(n+1)\frac{2}{1-\gamma^{n_0}}$ and use
 (\ref{LY-hole3}), (\ref{PFU}) to obtain a uniform the Lasota-Yorke inequality for $\Lh$ and $\LT$.\\
\end{proof}
\begin{lemma}\label{Le4}
Consider $\Lh: BV(I)\to BV(I)$. Then there are $e_{\eps}\in (0,1)$, $\varphi_{\eps}\in BV(I)$, a probability Borel measure $\nu_{\eps}$ and linear operators $\mathcal Q_{\eps}:BV(I)\to BV(I)$ such that
\begin{enumerate}
\item $e_{\eps}^{-1}\Lh=\varphi_{\eps}\otimes\nu_{\eps}+\mathcal Q_{\eps}$;
\item $\Lh\varphi_{\eps}=e_{\eps}\varphi_{\eps},\, \nu_{\eps}\Lh=e_{\eps}\nu_{\eps},\, \mathcal Q_{\eps}\varphi_{\eps}=0,\, \nu_{\eps}\mathcal Q_{\eps}=0$;
\item $\sum_{n=0}^{\infty}\sup_{\eps}||\mathcal Q_{\eps}^n||_{BV(I)}<\infty$;
\item $m(\varphi_{\eps})=1$ and $\sup_{\eps}||\varphi_{\eps}||_{BV(I)}<\infty$;
\item $\exists\, C>0$ such that
$$\eta_{\eps}:=\sup_{||\psi||_{BV(I)}\le 1}|\int_I(\LT-\Lh)\psi dm|\to 0\text{ as }\eps\to 0,$$
and
$$\eta_{\eps}\cdot ||(\LT-\Lh)\rho||_{BV(I)}\le C|\Delta_{\eps}|,$$
where $\Delta_{\eps}:=m((\LT-\Lh)\rho).$
\end{enumerate}
\end{lemma}
\begin{proof}
We have for any $\psi\in BV(I)$:
\begin{equation}\label{diff-hole}
(\LT-\Lh)\psi=\int_{\Omega_{\eps}}\LT(\psi {\bf 1}_{H_{\omega}})d\theta_{\eps}.
\end{equation}
Since $\psi$ is also in $L^{\infty}(m)$ we have
\begin{equation}\label{diff-hole1}
\eta_{\eps}=\sup_{||\psi||_{BV(I)}\le
1}|\int_{I}(\LT-\Lh)\psi dm|\le\int_{\Omega_{\eps}}m(H_{\omega})d\theta_{\eps}\le
m(H_{\eps}) \to 0\text{ as } \eps\to 0.
\end{equation}
Thus, by Lemma \ref{LeU} and (\ref{diff-hole1}), using the abstract perturbation result on the stability of spectrum of transfer operators \cite{KL1}, we obtain (1)-(4) of the lemma. Now by (\ref{diff-hole}), notice that
\begin{equation}\label{eta}
\eta_{\eps}\le A_{\eps}:=\int_{\Omega_{\eps}}m(H_{{\omega}})d\theta_{\eps},
\end{equation}
and
\begin{equation}\label{delta}
\Delta_{\eps}=\int_{\Omega_{\eps}}\mu(H_{{\omega}})d\theta_{\eps}.
\end{equation}
Moreover, using (\ref{LY-hole0}), condition {\bf (O3)}, and calling: (i) $K=\max (A\gamma, B)$; (ii) $\tilde{C}= 3||\rho||_{BV}$; (iii) $\rho_m:=\inf_{x\in I}\rho(x), x-m \ a.e.$,  we obtain
\begin{equation*}\label{cond5}
||(\LT-\Lh)\rho||_{BV(I)}=||\int_{\Omega_{\eps}}\LT(\rho{\bf{1}}_{H_{\omega}})d\theta_{\eps}||_{BV(I)}
\end{equation*}
\begin{equation}\label{cond5}
\le  K \int_{\Omega_{\eps}}\left(V(\rho{\bf 1}_{H_{\omega}})+||\rho{\bf 1}_{H_{\omega}}||_1\right)d\theta_{\eps}\\\le \frac{K \tilde{C}}{\rho_m}\frac{ \int_{\Omega_{\eps}}\mu(H_{\omega})d\theta_{\eps}}{{ A_{\eps}}}.
\end{equation}
Thus, by (\ref{eta}), (\ref{delta}) and (\ref{cond5}),
$$\eta_{\eps}\cdot ||(\LT-\Lh)\rho||_{BV(I)}\le \frac{K \tilde{C}}{\rho_m}\cdot \Delta_{\eps}.$$
\end{proof}
We are now ready to apply the abstract perturbation result of \cite{KL2}. Indeed, Lemma \ref{Le4} shows that all the conditions imposed in \cite{KL2} are satisfied by our random systems.
\begin{lemma}\label{Le5}
\text{ }
\begin{enumerate}
\item For sufficiently small $\eps$, $\Delta_{\eps}\not=0$.
\item If for each integer $k\ge 0$ the following limit:
$$q_k:=\lim_{\eps\to 0}q_{k,\eps}:=\lim_{\eps\to 0}\frac{m((\LT-\Lh)\Lh^k(\LT-\Lh)(\rho))}{\Delta_{\eps}}$$
exists, then 
$$\lim_{\eps\to 0}\frac{1-e_{\eps}}{\Delta_{\eps}}=1-\sum_{k=0}^{\infty}q_k.$$
\end{enumerate}
\end{lemma}
\begin{proof}
By condition {\bf (O3)}, for sufficiently small $\eps>0$, $\Delta_{\eps}>0$. Thus, Part (2) of the lemma follows from the abstract perturbation result of \cite{KL2}, since we will show in the proof of Theorem \ref{main1} that for each $k\ge0$ the limit in (2) exists.
\end{proof}

\begin{proof} (Proof of Theorem \ref{main1}) We start by computing $q_{k,\eps}$ and show that the limit in (2) of Lemma \ref{Le5} holds. Throughout the computation in \eqref{qke} below, we will use repeatedly the fact that for any measurable set $A$ and integrable function $f$ we have
${\bf 1}_A\LT f=\LT({\bf 1}_{T^{-1}A} f).$

Consequently, by (\ref{diff-hole}) and (\ref{it-hole}), we obtain
\begin{equation}\label{qke}
\begin{split}
&m((\LT-\Lh)\Lh^k(\LT-\Lh)(\rho))\\
&=\int_{\Omega_{\eps}}\int_{I}((\LT-\Lh)\Lh^k(\LT-\Lh)(\rho))dmd\theta_{\eps}(\omega)\\
&=\int_{\Omega_{\eps}}\int_{I}\LT({\bf 1}_{H_{\omega}}\cdot\Lh^k(\LT-\Lh)(\rho))dmd\theta_{\eps}(\omega)\\
&=\int_{\Omega_{\eps}}\int_{I}{\bf 1}_{H_{\omega}}\cdot\Lh^k(\LT-\Lh)(\rho))dmd\theta_{\eps}(\omega)\\
&=\int_{\Omega_{\eps}}\int_{I}\int_{\bar\Omega_{\eps}}{\bf 1}_{H_{\omega}}\cdot   \LT^k {\bf 1}_{X_{\omega_1}\cap \cdots\cap T^{-(k-1)}X_{\omega_k}}\cdot(\LT-\Lh)(\rho))d\theta_{\eps}(\bar\omega)dm d\theta_{\eps}(\omega)\\
&=\int_{\Omega_{\eps}}\int_{I}\int_{\bar\Omega_{\eps}} \LT^k( {\bf 1}_{T^{-k}H_{\omega}}\cdot {\bf 1}_{X_{\omega_1}\cap \cdots\cap T^{-(k-1)}X_{\omega_k}}\cdot(\LT-\Lh)(\rho))d\theta_{\eps}(\bar\omega)dm d\theta_{\eps}(\omega)\\
&=\int_{\Omega_{\eps}}\int_{\bar\Omega_{\eps}} \int_{I} {\bf 1}_{T^{-k}H_{\omega}}\cdot {\bf 1}_{X_{\omega_1}\cap \cdots\cap T^{-(k-1)}X_{\omega_k}}\cdot(\LT-\Lh)(\rho))dm d\theta_{\eps}(\bar\omega)d\theta_{\eps}(\omega)\\
&=\int_{\Omega_{\eps}}\int_{\bar\Omega_{\eps}} \int_{I}\int_{\Omega_{\eps}} {\bf 1}_{T^{-k}H_{\omega}}\cdot {\bf 1}_{X_{\omega_1}\cap \cdots\cap T^{-(k-1)}X_{\omega_k}}\cdot(\LT{\bf 1}_{H_{\omega'}}\rho )dmd\theta_{\eps}(\bar\omega)d\theta_{\eps}(\omega)d\theta_{\eps}(\omega')\\
&=\int_{\Omega_{\eps}}\int_{\bar\Omega_{\eps}} \int_{I}\int_{\Omega_{\eps}} \LT({\bf 1}_{T^{-(k+1)}H_{\omega}}\cdot {\bf 1}_{T^{-1}X_{\omega_1}\cap \cdots\cap T^{-k}X_{\omega_k}}\cdot{\bf 1}_{H_{\omega'}}\rho)dmd\theta_{\eps}(\bar\omega)d\theta_{\eps}(\omega)d\theta_{\eps}(\omega')\\
&=\int_{\Omega_{\eps}}\int_{\bar\Omega_{\eps}} \int_{\Omega_{\eps}}\mu\left([T^{-(k+1)}H_{\omega}\cap{T^{-1}X_{\omega_1}\cap \cdots\cap T^{-k}X_{\omega_k}}]\cap H_{\omega'} \right) d\theta_{\eps}(\omega')d\theta_{\eps}(\bar\omega)d\theta_{\eps}(\omega).
\end{split}
\end{equation}
Now, if $z$ is not a periodic point of $T$, since $T$ is continuous at $z$, for sufficiently small $\eps>0$,
$$[{T^{-1}X_{\omega_1}\cap \cdots\cap T^{-k}X_{\omega_k}} \cap T^{-(k+1)}H_{\omega}]\cap H_{\omega'}=\emptyset.$$
Therefore, (\ref{qke}) implies $q_{k,\eps}=q_k=0$. This proves (1) of the theorem.

\bigskip

To prove (2) of the theorem, since $z$ is a periodic point of minimal period $p$ and the density $\rho$ is essentially bounded from below, the quantity we have to compute reduces to
\begin{equation}
\frac{\int \int m(T^{-p}H_{\omega}\cap H_{\omega'})d\theta_{\epsilon}(\omega)d\theta_{\epsilon}(\omega')}{\int m(H_{\omega})d\theta_{\epsilon}(\omega)}.
\end{equation}
Since $T^p$ is continuous around $z$ we can suppose that, by taking $\epsilon$ small enough, all the $H_{\omega}$, $|\omega|\le \epsilon$, will be contained in a neighborhood of $z$ where $T^p$ is continuous and monotone. This in particular implies the following fact that we will use later on: if we call $H_{\omega,p}$ the unique connected component of $T^{-p}H_{\omega}$ which contains $z$, it will be the only connected component of $T^{-p}H_{\omega}$ which intersects any other $H_{\omega'}$.
By the differentiability assumption on $T^p$ and the local change of variable we have
\begin{equation}\label{per1}
\begin{split}
m(T^{-p}H_{\omega}\cap H_{\omega'})=\int_{H_{\omega,p}\cap H_{\omega'}}dm=\int_{T^p(H_{\omega,p}\cap H_{\omega'})}|DT^p(y)|^{-1}dm(y)\\
=\int_{H_{\omega}\cap T^pH_{\omega'}}|DT^p(y)|^{-1}dm(y).
\end{split}
\end{equation}
Equation \eqref{per1} gives immediately the upper bound
\begin{equation}\label{upper}
m(T^{-p}H_{\omega}\cap H_{\omega'})\le \sup_{H_{\omega}}|DT^p|^{-1} \ m(H_{\omega})\le \sup_{H_{\eps}}|DT^p|^{-1} \ m(H_{\omega}).
\end{equation}
By \eqref{upper} and the continuity of $DT^p$ at $z$ we obtain
\begin{equation}\label{per2}
\lim_{\eps\to0}\frac{\int \int m(T^{-p}H_{\omega}\cap H_{\omega'})d\theta_{\epsilon}(\omega)d\theta_{\epsilon}(\omega')}{\int m(H_{\omega})d\theta_{\epsilon}(\omega)}\le |DT^p|^{-1}(z).
\end{equation}
For the lower bound, since $T^p$ is expanding, we have
\begin{equation}\label{lower1}
\begin{split}
m(T^{-p}H_{\omega}\cap H_{\omega'})&\ge \int_{H_{\omega}\cap H_{\omega'}}|DT^p(y)|^{-1}dm(y)\\
&=\int_{H_{\omega}}|DT^p(y)|^{-1}dm(y)-\int_{H_{\omega}/H_{\omega'}}|DT^p(y)|^{-1}dm(y)\\
&\ge \inf_{H_{\eps}}|DT^p|^{-1} \ m(H_{\omega})-\int_{H_{\omega}/H_{\omega'}}|DT^p(y)|^{-1}dm(y).
\end{split}
\end{equation}
Therefore,
\begin{equation}\label{per3}
\begin{split}
\frac{\int \int m(T^{-p}H_{\omega}\cap H_{\omega'})d\theta_{\epsilon}(\omega)d\theta_{\epsilon}(\omega')}{\int m(H_{\omega})d\theta_{\epsilon}(\omega)}&\ge \int \int\inf_{H_{\eps}}|DT^p|^{-1}d\theta_{\epsilon}(\omega)d\theta_{\epsilon}(\omega')\\
&-\frac{\int \int \int_{H_{\omega}/H_{\omega'}}|DT^p(y)|^{-1}dm(y)d\theta_{\epsilon}(\omega)d\theta_{\epsilon}(\omega')}{\int m(H_{\omega})d\theta_{\epsilon}(\omega)}.
\end{split}
\end{equation}
Let us consider the second expression on the right hand side of \eqref{per3} and show that it vanishes when $\eps$ goes to zero. First, we have 
$$\int_{H_{\omega}/H_{\omega'}}|DT^p(y)|^{-1}dm(y)\le \sup_{H_{\omega}}|DT^p|^{-1} \ m(H_{\omega}\Delta H_{\omega'}).
$$
Now, let us call $G_{\omega}$ the complement of $H_{\omega}$ in $H_{\eps}$; we immediately have by the condition {\bf (C)}: $m(H_{\omega}\cap H_{\omega'})=m(H_{\omega'})-m(H_{\omega'}\cap G_{\omega})\ge (\eps-\eps^{\upsilon})-\eps^{\upsilon}\ge \eps-2 \eps^{\upsilon}$. This  implies that $m(H_{\omega}/H_{\omega'})=m(H_{\omega})-m(H_{\omega}\cap H_{\omega'})\le 2 \eps^{\upsilon}$ and therefore $m(H_{\omega}\Delta H_{\omega'})\le 4\eps^{\upsilon}.$ Let us call $F_{\eps}:=\{\omega \in \Omega_{\eps}; \ m(H_{\omega})\in (\eps-\eps^{\upsilon}, \eps)\}$; $Q_{\eps}:=\{(\omega, \omega')\in \Omega_{\eps} \times \Omega_{\eps}; \ m(H_{\omega}\Delta H_{\omega'})\}\le 4\eps^{\upsilon}$. Notice that $Q_{\eps}\supset F_{\eps}\times F_{\eps}$ and recall that $\theta_{\eps}(F_{\eps})\ge 1-\eps^{\upsilon}$.\\
Then
\begin{equation*}
\begin{split}
&\frac{\int \int \int_{H_{\omega}/H_{\omega'}}|DT^p(y)|^{-1}dm(y)d\theta_{\epsilon}(\omega)d\theta_{\epsilon}(\omega')}{\int m(H_{\omega})d\theta_{\epsilon}(\omega)}\\
&\le\frac{\int \int m(H_{\omega}\Delta H_{\omega'})d\theta_{\epsilon}(\omega)d\theta_{\epsilon}(\omega')}{\int m(H_{\omega})d\theta_{\epsilon}(\omega)}\\
&\le \frac{1}{\int m(H_{\omega})d\theta_{\epsilon}(\omega)}[\int \int_{Q_{\epsilon}} m(H_{\omega}\Delta H_{\omega'})d\theta_{\epsilon}(\omega)d\theta_{\epsilon}(\omega')\\
&+\int \int_{Q_{\eps}^c} m(H_{\omega}\Delta H_{\omega'})d\theta_{\epsilon}(\omega)d\theta_{\epsilon}(\omega')]\\
&\le \frac{1}{\int m(H_{\omega})d\theta_{\epsilon}(\omega)}[4\epsilon^{\upsilon}+\theta_{\eps}^2(F_{\eps}^c))]\\
& \le \frac{1}{\int_{F_{\eps}} m(H_{\omega})d\theta_{\epsilon}(\omega)}[4\epsilon^{\upsilon}+\eps^{2\upsilon}]\le \frac{4\epsilon^{\upsilon}+\eps^{2\upsilon}}{(\epsilon-\epsilon^{\upsilon})(1-\epsilon^{\upsilon})}
\end{split}
\end{equation*}
which goes to zero when $\eps$ tends to $0$. Consequently, taking the limit as $\eps$ goes to $0$ in \eqref{per3}, by the continuity of $DT^p$ at $z$, we obtain
\begin{equation}\label{per4}
\lim_{\eps\to0}\frac{\int \int m(T^{-p}H_{\omega}\cap H_{\omega'})d\theta_{\epsilon}(\omega)d\theta_{\epsilon}(\omega')}{\int m(H_{\omega})d\theta_{\epsilon}(\omega)}\ge |DT^p|^{-1}(z).
\end{equation}
Thus (2) of the theorem follows by \eqref{per2} and \eqref{per4}.
\end{proof}
\section{Metastability of randomly perturbed maps}\label{Meta}
In this section we study random perturbations of Lasota-Yorke maps. We assume that the initial system admits exactly two ergodic invariant densities. Then under random perturbations which allow leakage of mass through random holes, we will show that the system admits a unique stationary density. Our goal is to show that such a  density  can be approximated in the $L^1$-norm by a particular convex combination of the two invariant densities of the initial system. A deterministic Lasota-Yorke system of this type was studied in \cite{GHW}. Deterministic intermittent systems of this type were studied in \cite{BV}. Other recent results on metastable random dynamical systems have been obtained in \cite{FS, GBE}. We will establish a close link between the escape rate formulae which we obtained in the previous section and invariant densities of randomly perturbed metastable systems. We first introduce the class of maps of this section.
\subsection{The initial system}
Let $T: I\to I$ be a map which satisfies the following conditions:\\
\noindent {\bf (A1)} There exists a partition of $I$, which consists of intervals $\{I_i\}_{i=1}^{q}$, $I_i\cap I_j=\emptyset$ for $ i\not= j$, $\bar I_i :=[c_{i,0},c_{i+1,0}]$ and there exists $\delta>0$  such that $T_{i,0}:= T|_{(c_{i,0},c_{i+1,0})}$ is $C^2$ which extends to a $C^2$ function $\bar T_{i,0}$ on a neighbourhood $[c_{i,0}-\delta,c_{i+1,0}+\delta]$ of $\bar I_i $ ;\\
\noindent {\bf (A2)} $\inf_{x\in I\setminus\mathcal C_0}|T'(x)|\ge\beta_0^{-1}>2$, where $\mathcal C_0=\{c_{i,0}\}_{i=1}^{q}$.\\
\noindent {\bf (A3)} $\exists$ $b$ in the interior of $I$ such that $T|_{I_*}\subseteq I_*$, where $*\in\{l,r\}$, $I_*$ is an interval such that $I_l\cup I_r=I$ and $I_l\cap I _r=\{b\}$.\\
\noindent {\bf (A4)} Let $H_0:= T^{-1}\{b\}\setminus\{b\}$. We call $H_0$ the set of \textit{infinitesimal holes} and we assume that for every $n\ge 1$, $(T^n\mathcal C_0)\cap H_0=\emptyset.$\\
\noindent {\bf (A5)} $T$ admits exactly two ergodic a.c.i.ms $\mu_*$, such that each $\mu_*$ is supported on $I_*$ and the corresponding density $\rho_*$ is positive at each of the points of $H_0 \cap I_*$

\bigskip
\begin{remark}
As shown in \cite{GHW}, assumption {\bf (A4)} guarantees that the invariant densities of the two ergodic measures are continuous at each of the infinitesimal holes.\\ 
Assumption {\bf (A2)}, more precisely the fact that $\beta_0^{-1}$ is strictly bigger than $2$ instead of $1$, is sufficient to get the uniform Lasota-Yorke inequality of Proposition \ref{prop1}, as is explained in section 4.2 of \cite{GHW}. Finally, we consider
$T$ to be bi-valued at the points $c_{i,0}$ of discontinuity in $\mathcal C_0$ and
we take  $T(c_{i,0})$ be both values obtained as $x$ tends to  $c_{i,0}$ from either side,
and $T(c_{i,0}\pm)$ will be the corresponding right and left limits.
\end{remark}
An example of a map $T_0$ satisfying the above assumption is shown in Figure 2.
\begin{figure}[htbp] 
   \centering
   \includegraphics[width=2.5in]{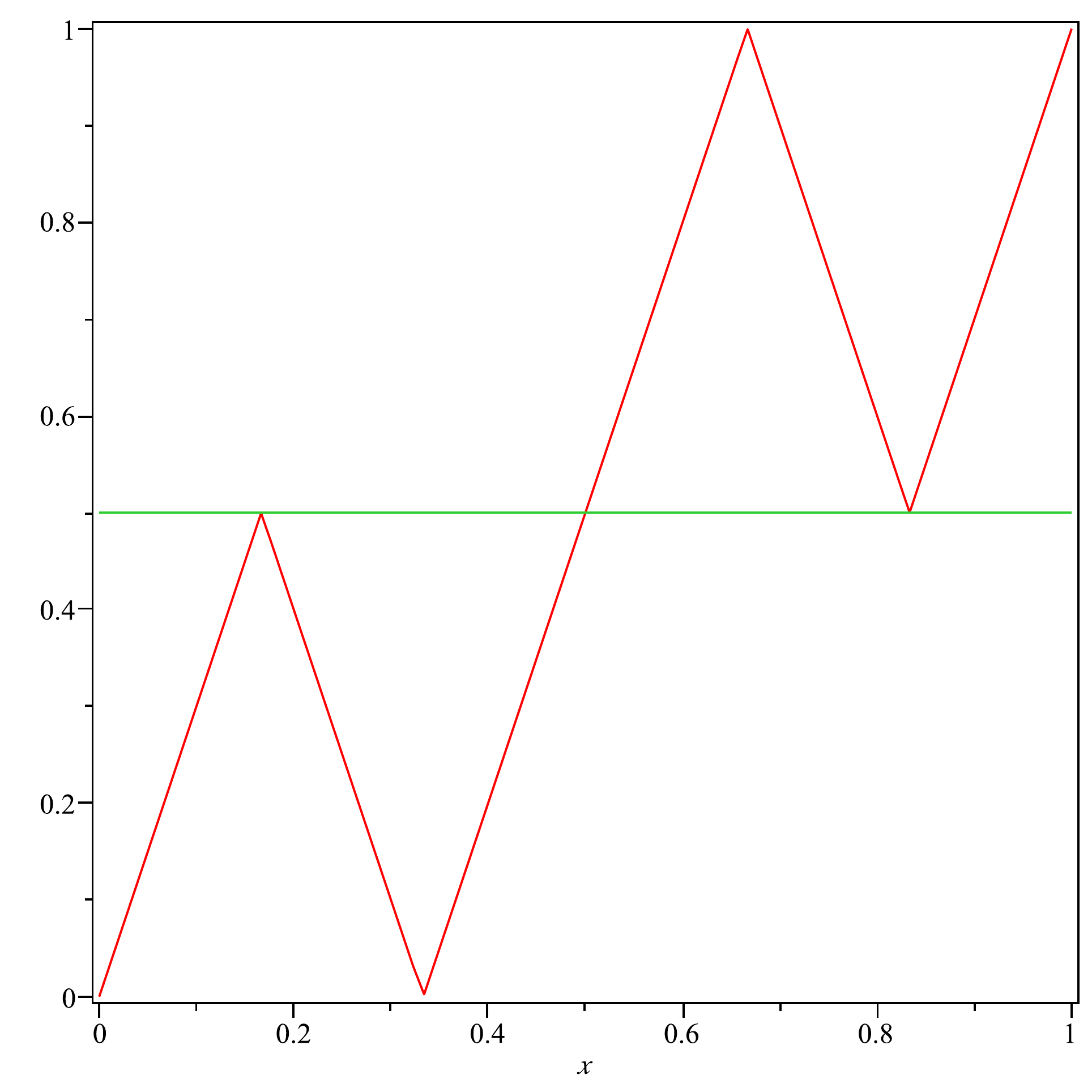}
   \caption{A typical example of the initial system $T$.}
\end{figure}
We denote the transfer operator (Perron-Frobenius) \cite{Ba, BG} associated with $T$ by $\LT$.
\subsection{Random perturbations on a finite noise space}
As in the previous section, let $(\omega_k)_{k\in\mathbb N}$ be an i.i.d.  stochastic process with values in $\Omega_{\eps}$ and with probability distribution $\theta_{\eps}$. However, we assume in this section that 
$\Omega_{\eps} =\{{\bf s}_0, {\bf s}_1, \dots,{\bf s}_{L}\}$ is a finite set, and $\theta_{\eps}$ is an atomic measure on $\Omega_{\eps}$\footnote{In fact, the assumption that the set $\Omega_{\eps}$ is a finite set is not need in some of the proofs of this section. In particular in Lemmas \ref{Le1} and \ref{Le2}. However, this assumption is needed to be able to identify the jumps in the stationary density $\rho_{\eps}$ of the random system. It seems that the assumption of this section that $\Omega_{\eps}$ is a finite set is the best one can do when dealing with approximating the invariant density of the random metastable system. See Remark \ref{count} where we explain why the problem becomes intractable if one relaxes this assumption even to the case where $\Omega_{\eps}$ is countable but not finite.}. We associate with each $\omega\in\Omega_{\eps}$ a map $T_{\omega}: I\to I$ with $T_0 = T_{{\bf s}_0}$, $T_{\eps}=T_{{\bf s}_{L}}$, and we consider the random orbit starting from the point $x$ and generated by the realization $\underline{\omega}_n=(\omega_1,\omega_2,\cdots,\omega_n)$, defined as : $T_{\underline{\omega}_n}:=T_{\omega_n}\circ\cdots\circ T_{\omega_1}(x)$ ({\em random transformations}). This defines a Markov process  $\mathfrak T_{\eps}$ with transition function
$$\mathbb P (x, A)=\int_{\Omega_{\eps}}\mathbf{1}_{A}(T_{\omega}(x))d\theta_{\eps}(\omega),$$
where $A\in\mathfrak B(I)$, $x\in I$ and $\mathbf{1}_{A}$ is the indicator function of a set $A$. The transition function  induces an operator ${\mathcal U}_{\eps}^*$ which acts on measures $\Lambda$ on $(I,\mathfrak B(I))$ as:
$${\mathcal U}_{\eps}^*\Lambda(A)=\int_{I}\int_{\Omega_{\eps}}\mathbf{1}_{A}(T_{\omega}(x))d\theta_{\eps}(\omega)d\Lambda(x)=\int_{I}{\mathcal U}_{\eps}{\bf 1}_A(x)d\mu_{\eps}(x),$$ where ${\mathcal U}_{\eps}$ is the random evolution operator acting on $L^{\infty}_m$ functions $g$:
\begin{equation}\label{koop}
{\mathcal U}_{\eps}g=\int_{\Omega_{\eps}}g\circ T_{\omega}d\theta_{\eps}(\omega).
\end{equation}
A measure $\mu_{\eps}$ on $(I,\mathfrak B(I))$ is called a $\mathfrak T_{\eps}$-stationary measure  if and only if, for any $A\in\mathfrak B(I)$,
\begin{equation}\label{stationary}
{\mathcal U}_{\eps}^*\mu_{\eps}(A)=\mu_{\eps}(A).
\end{equation}
We are interested in studying the metastability of $\mathfrak T_{\eps}$-stationary measures which are absolutely continuous with respect to $m$: let us call them {\em acsm}. By \eqref{koop}, one can define the transfer operator $\mathcal L_{\eps}$ (Perron-Frobenius) acting on $L^{1}(I,\mathfrak B(I), m)$ by:
\begin{equation}\label{RPF}
(\Le f)(x)=\int_{\Omega_{\eps}}\Lw f (x)d\theta_{\eps}(\omega),
\end{equation}
which satisfies the duality condition
\begin{equation}\label{du}
\int_I g \Le f dm = \int_I {\mathcal U}_{\eps}g f dm
\end{equation}
where $g$ is in some subset of $L^{\infty}_m$ and $\Lw$ is the transfer operator associated with $T_{\omega}$. In the present setting $g\in BV(I)$; i.e.,  a function of bounded variation.
It is well know that $\mu_{\eps}:=\rho_{\eps}m$ is a $\mathfrak T_{\eps}$-acsm if and only if $\Le\rho_{\eps}=\rho_{\eps}$; i.e., $\rho_{\eps}$ is a $\mathfrak T_{\eps}$-invariant density. For each $\omega\in\Omega_{\eps}$, we assume that $T_{\omega}$ satisfies the following conditions:\\

\noindent {\bf (B1)}
There exists a  partition of $I$, which consists of intervals $\{I_{i,\omega}\}_{i=1}^{q}$, $I_{i,\omega}\cap I_{j,\omega}=\emptyset$ for $ i\not= j$, $\bar I_{i,\omega} :=[c_{i,\omega},c_{i+1,\omega}]$ such that\\
(i) for each $i$ and  for all $0\le \omega\le \eps$, $\eps$ small enough,    we have that (the quantity $\delta$ was introduced in the assumption {\bf (A1)} above):\\  $[c_{i,0}+\delta,c_{i+1,0}-\delta]\subset [c_{i,\omega},c_{i+1,\omega}]\subset [c_{i,0}-\delta,c_{i+1,0}+\delta]$; whenever $c_{i,0}\not=0$ and  $c_{i+1},0\not=1$. In this way we have  established  a one-to-one correspondence between the unperturbed and the perturbed boundary points of $I_i$ and $I_{i, \omega}$. If $c_{i,0}$ and $c_{i, \omega}$ are two such (left or right)  {\em corresponding points} we will ask that $\forall i$ and $\forall \omega\in \Omega_{\eps}$:\\
 \begin{equation}\label{C1}
 \lim_{\eps\rightarrow 0} |c_{i,0}-c_{i, \omega}|=0 \ \  \mbox{({\em uniform collapsing of boundary points})}
 \end{equation}\\
(ii) The map $T_{\omega}$ is locally injective over the closed intervals $\overline{I}_{i,\omega}$, of class $C^2$ in their interiors, and expanding with $\inf_{\omega,x}|T_{\omega}'x|\ge \beta>2$. Moreover,  $\forall \omega \in \Omega $, $\forall i=1,\cdots, q$  and $\forall x \in [c_{i,0}-\delta,c_{i+1,0}+\delta]$ we have
 \begin{equation}\label{C2}
 \lim_{\eps\rightarrow 0} |\bar{T}_{i,0}(x)-T_{i, \omega}(x)|=0 \ \  \mbox{({\em uniform convergence of maps})}
 \end{equation}\\

\noindent {\bf (B2)}  Boundary condition:\\
(i) if $b\notin \mathcal C_0$, then $T(b)=b$ and for all $\omega> 0$, $T_{\omega}(b)=b$; \\
(ii) if $b\in \mathcal C_0$, then $ T(b-)<b< T(b+)$ and for all $\omega>0$, $b\in \mathcal C_{\omega}$, where $\mathcal C_{\omega}=\{c_{i,\omega}\}_{i=1}^{q}$.\\

The last condition, as explained in section 2.4 of \cite{GHW} does not allow the appearance of other holes, besides those around the infinitesimal holes, in the neighborhood of $b$. These conditions are satisfied when the noise is, for instance, additive: $T_{\omega}(x)=T(x)+\omega$. It is always well defined on the circle by taking the mod-$1$ unfolding; however, on the interval we should consider maps $T$ for which the images $|T_{\omega}(x)|$ lie in the unit interval for all $\omega\in \Omega_{\eps}$. In both cases the intervals of local monotonicity will be always the same and the differences between the perturbed and unperturbed images are uniformly bounded by $\eps$. Another example of a $T_\omega$ is shown in Figure 3.
\begin{figure}[htbp] 
   \centering
   \includegraphics[width=2.5in]{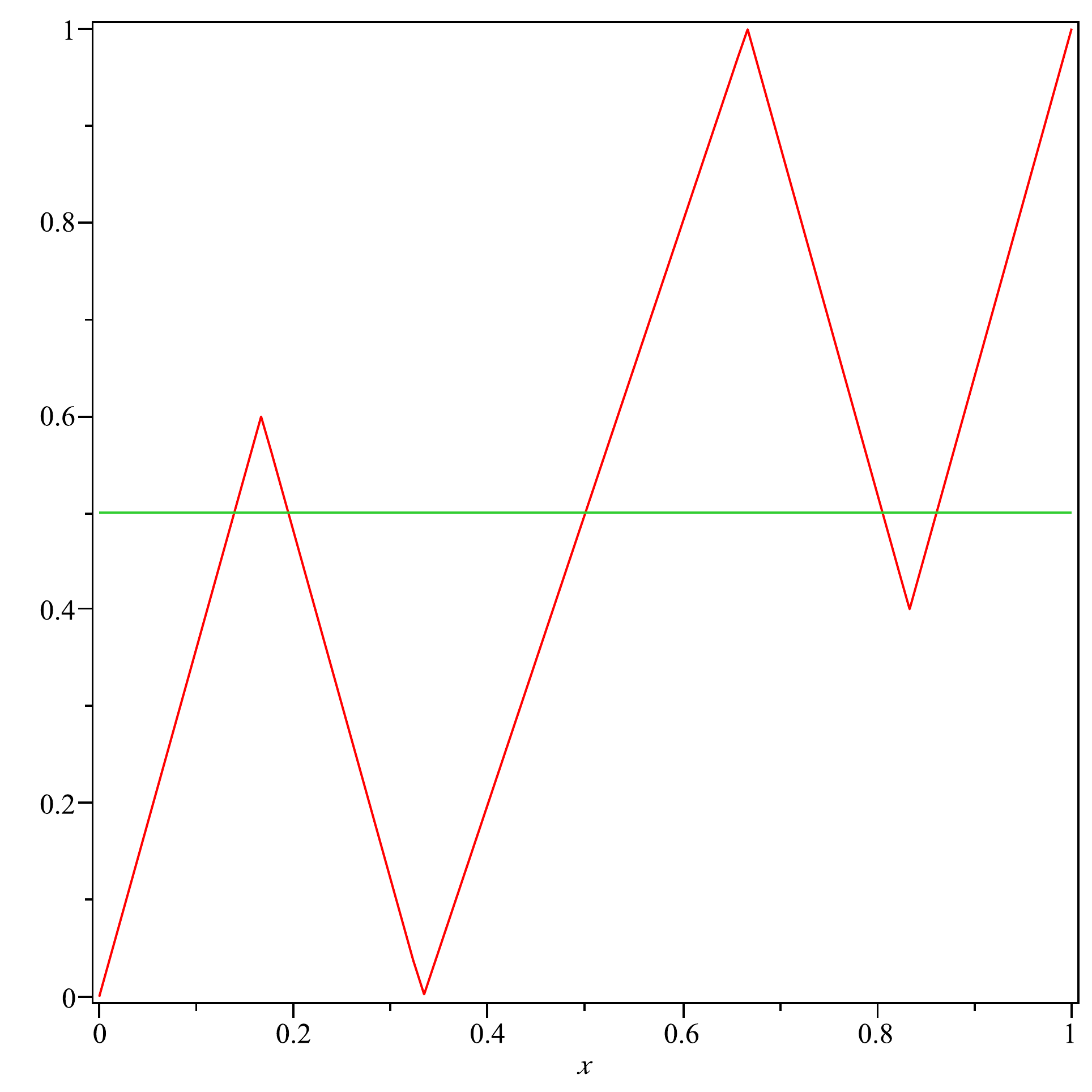}
   \caption{A typical graph of one of the constituent maps $T_{\omega}$ of the metastable random map $\mathfrak T_{\eps}$ corresponding to the initial system shown in Figure 2.}
\end{figure}

Under the above assumptions it is well known that (see for instance \cite{BaYo})
\begin{proposition}\label{prop1}
For sufficiently small $\eps>0$, there exists a $\eta\in (0,1)$ and a $B\in (0,\infty)$, such that for any $f\in BV(I)$ and $n\ge1$, we have
\begin{equation}\label{LY}\left\|\mathcal{L}_{\bullet}^n f\right\|_{BV}\le\eta^n\left\|f\right\|_{BV} +B||f||_1,\end{equation}
where $\bullet\in\{0, \eps\}$; i.e., both $\LT$ and $\Le$ satisfy a uniform Lasota-Yorke inequality as operators on $BV(I)$.
\end{proposition}
Among other things, Proposition \ref{prop1} implies that $1$ is an eigenvalue of $\Le$ and $\mathfrak T_{\eps}$ admits a stationary density $\rho_{\eps}$. We further assume assume that:\\

\noindent {\bf (R1)} $\rho_{\eps}$ is the unique invariant density of $\mathfrak T_{\eps}$ \ ({\em ergodicity}).\\

\begin{remark}\label{Gora}
Since in this section we assumed $\Omega_{\eps}$ to be finite, then condition {\bf (R1)} is satisfied if one of the maps $T_{\omega}$ has a unique acim. See Corollary 2 of \cite{Go}. For instance, condition {\bf (R1)} is satisfied whenever the deterministic map $T_{\eps}$ is topologically mixing. 
\end{remark}
\subsection{Random holes in the perturbed system} \label{Meta} We are interested in
perturbations of $T$ which produce ``leakage" of mass from $I_l$ to
$I_r$ and vice versa. For this purpose, for each
$\omega\in\Omega_{\eps}$ we define the following sets:
$$H_{l,\omega}:=I_l\cap T^{-1}_{\omega}(I_r)$$
and
$$H_{r,\omega}:=I_r\cap T^{-1}_{\omega}(I_l).$$
The sets $H_{l,\omega}$ and $H_{r,\omega}$ are called the ``left hole" and the ``right hole", respectively, of the map $T_{\omega}$ . Thus, when $T_{\omega}$ allows leakage of mass from $I_l$ to $I_r$, this leakage occurs when orbits of $T_{\omega}$ fall in the set $H_{l,\omega}$. Similarly, when $T_{\omega}$ allows leakage of mass from $I_r$ to $I_l$, this leakage occurs when orbits of $T_{\omega}$ fall in the set $H_{r,\omega}$.\\

For each of the left and right `random open systems', we define a transfer operator which will be used to find the exponential escape rates from $I_l$ and $I_r$. For this purpose we set
$$X_{l,\omega}:=I_l\setminus H_{l,\omega},$$
and define for $f\in L^1(I_l,\mathfrak{B}(I_l), \frac{m}{b})$
\begin{equation}\label{OPH}
\LH f(x):=\int_{\Omega_\eps}\LTl (f\mathbf{1}_{X_{l,\omega}})(x)d\theta_\eps,
\end{equation}
where $\LTl$ is the transfer operator associated with the map $T_{|_{I_l}}:=T_{l}:I_l\to I_l$. In a similar way, we can define a transfer operator $\hat{\mathcal{L}}_{r,\eps}$ associated with the right `random open system'.\\
\subsection{Statement of the main result of section \ref{Meta}} Using Proposition \ref{prop1} and compactness arguments, it is well known that $\rho_{\eps}$ converges in the $L^1$-norm to a convex combination of $\rho_l$ and $\rho_r$. Our main goal in this part of the paper is to explicitly identify the weights in the convex combination. Since the acsm is ergodic, at least one of the two holes will have positive Lebsegue measure; we will suppose without loss of generality that the left hole will have this property. It will also carry positive mass for the left unperturbed invariant measure $\mu_l$ since we assumed that the densities of $\mu_l$ and $\mu_r$ are positive in the neighborhood of the respective holes. We therefore define the \textit{limiting averaged holes ratio (l.a.h.r.)} by
\begin{equation}\label{lahr}
l.a.h.r. :=\lim_{\eps\to 0}\frac{\int_{\Omega_\eps}\mu_r(H_{r,\omega})d\theta_{\eps}(\omega)}{\int_{\Omega_\eps}\mu_l(H_{l,\omega})d\theta_{\eps}(\omega)}
\end{equation}
whenever the limit exists.
We obtain the following theorem:
\begin{theorem}\label{main}
Let $\rho_{\eps}$ be the unique $\mathfrak T_{\eps}$-invariant density. Suppose the $l.a.h.r$ exists then
$$\lim_{\eps\to 0}||\rho_\eps -\rho_0||_1=0,$$
where $\rho_0=\alpha \rho_{l}+(1-\alpha)\rho_r$ and $l.a.h.r=\frac{\alpha}{1-\alpha}$.
\end{theorem}
Theorem \ref{main1} will give us  additional information about the ratio $\frac{\alpha}{1-\alpha}$ which appears in Theorem \ref{main} under the additional assumption:\\

{\bf (B4)} The map $T$ is continuous on $h_*$ and $\forall \omega,\ h_*\in H_{*, \omega}.$ Moreover, $(I_*, \mu_*, T_*)$, $*\in\{l,r\}$, is mixing.

\begin{corollary}\label{Co1}
 Assume condition {\bf (B4)}. Then
 \begin{enumerate}
 \item For sufficiently small $\eps>0$, there exists an $0<e_{*,\eps}<1$, and a $g_{*,\eps}(x)> 0$, with $\int_{I_*}g_{*,\eps}dm=1$,  such that
$$\hat{\mathcal{L}}_{*,\eps}g_{*,\eps}=e_{*,\eps}g_{*,\eps}.$$
 \item Moreover,
$$\lim_{\eps\to 0}\frac{1-e_{l,\eps}}{1-e_{r,\eps}}=\frac{\alpha}{1-\alpha}.$$
\end{enumerate}
\end{corollary}
\begin{remark} As it was observed in \cite{GHW}, and this remains true in our case as well, the limit (\ref{lahr}) surely  exists whenever the perturbations
open up holes whose size is  first order in $\eps$. \\We also note that condition {\bf (B4)} is satisfied, for instance, when the map $T$ is continuous in the neighborhood of the infinitesimal holes and we perturb it with an additive noise.\end{remark}
\subsection{Technical lemmas and the proof of Theorem \ref{main}}
\begin{lemma}\label{Le1}
We have $\int_{\Omega_\eps}\mu_{\eps}(H_{l,\omega})d\theta_{\eps}(\omega)=\int_{\Omega_\eps}\mu_{\eps}(H_{r,\omega})d\theta_{\eps}(\omega)$, where $d\mu_{\eps}=\rho_{\eps}dm$.
\end{lemma}
\begin{proof}
We have
\begin{equation}\label{1}
\mu_{\eps}(I_l)=\int_{\Omega_\eps}\mu_{\eps}(I_l)d\theta_{\eps}(\omega) =\int_{\Omega_\eps}\mu_{\eps}(X_{l,\omega})d\theta_{\eps}(\omega)+\int_{\Omega_\eps}\mu_{\eps}(H_{l,\omega})d\theta_{\eps}(\omega).
\end{equation}
Since $T^{-1}_{\omega}(I_l)=X_{l,\omega}\cup H_{r,\omega}$, we also have
\begin{equation}\label{2}
\int_{\Omega_{\eps}}\mu_{\eps}(T^{-1}_{\omega} I_l)d\theta_{\eps}(\omega)=\int_{\Omega_\eps}\mu_{\eps}(X_{l,\omega})d\theta_{\eps}(\omega)+\int_{\Omega_\eps}\mu_{\eps}(H_{r,\omega})d\theta_{\eps}.
\end{equation}
Using (\ref{1}), (\ref{2}) and that $\mu_{\eps}$ is $\mathfrak T_{\eps}$-invariant the lemma follows.
\end{proof}
\begin{lemma}\label{Le2}
There exists a constant $K$, depending only on the map $T$, such that for $0\le n<\infty$ there exists $\sigma$, depending on $n$, and $\eps_{\sigma}$, depending on $\sigma$, such that $\forall \eps\le \eps_{\sigma}$ and $\tilde{\sigma}\le \sigma$, we have
\begin{enumerate}
\item $\mathcal L^n_{\eps} 1$ is a $C^1$-function on $I_{\tilde\sigma}=[h_* -\tilde{\sigma}, h_*+\tilde{\sigma}]$, $*\in\{l,r\}$;
\item for any $x,y\in I_{\tilde\sigma}$, we have 
$$|\mathcal{L}_{\eps}^n1(x)-\mathcal{L}_{\eps}^n1(y)|\le K\cdot |x-y|;$$ in particular,
$$|\mathcal{L}_{\eps}^n1(x)-\mathcal{L}_{\eps}^n1(y)|\le K\cdot \tilde{\sigma}.$$
\end{enumerate}
\end{lemma}
\begin{proof}
By condition (A4) we have for every $n>0$, $(T^n_0 C_0)\cap H_0=\emptyset.$ In particular, for a fixed $n\in\mathbb N$, we can choose $\sigma'>0$ so small so that, for $0\le k\le n$, $(T_0^kC_0)\cap (h_*-\sigma',h_*+\sigma')=\emptyset$. We now prove that there is a possibly smaller subinterval $(h_*-\sigma,h_*+\sigma)$, $\sigma\le \sigma'$ upon which (1) and (2) of the lemma are satisfied. Let us give the explicit form of this operator.

The iterates of ${\mathcal L}_{\eps}$ are given by:
$$
({\mathcal L}_{\eps}^n 1)=\int\cdots\int {\mathcal L}_{\omega_1}\left( {\mathcal L}_{\omega_2} \left(\cdots {\mathcal L}_{\omega_n}\left(1\right)\right)\right)d\theta_\eps(\omega_1)\cdots d\theta_\eps(\omega_n).
$$

Let us call $I_{l,\omega_j}$ the $l$-domain of injectivity of the map $T_{\omega_j}$ and call $T^{-1}_{l, \omega_j}$ the inverse of $T_{\omega_j}$ restricted to $I_{l,\omega_j}$. By $T_{l,0}$ we denote the restriction of the unperturbed map to its $l$-th interval of monotonicity; finally  for simplicity of notation we drop the suffix $l,r$ from the intervals of the partition.  We have:
$$
\Psi_{\omega_n,\cdots,\omega_1}(x):={\mathcal L}_{\omega_n}\left( {\mathcal L}_{\omega_{n-1}} \left(\cdots{\mathcal L}_{\omega_1}\left(1\right)\right)\right)(x)=
$$
$$
\sum_{k_1,\ldots,k_n}\frac{1}{|D(T_{\omega_n}\circ\cdots\circ T_{\omega_1})((T^{-1}_{k_1, \omega_1}\circ\cdots\circ T^{-1}_{k_n, \omega_n})(x))|}\times {\bf 1}_{\Omega_{\omega_1,\cdots,\omega_n}^{k_1,\cdots, k_n})}(x).$$
The sets
\begin{equation*}
\begin{split}
&\Omega_{\omega_1,\cdots,\omega_n}^{k_1,\cdots, k_n}:=\\
&T^{-1}_{k_1, \omega_1}\circ\cdots\circ T^{-1}_{k_{n-1}, \omega_{n-1}}I_{k_n,\omega_n}\cap T^{-1}_{k_1, \omega_1}\circ\cdots\cap T^{-1}_{k_1, \omega_1} I_{k_{2},\omega_{2}}\cap I_{k_1,\omega_1}
\end{split}
\end{equation*}
are intervals and they give a partition mod-$0$ of $I=[0,1]$; moreover the image
\begin{equation*}
\begin{split}
&H_{\omega_1,\cdots,\omega_n}^{k_1,\cdots, k_n}:=T_{\omega_n}\circ\cdots\circ T_{\omega_1}\Omega_{\omega_1,\cdots,\omega_n}^{k_1,\cdots, k_n}\\
&=T_{\omega_n}I_{k_n,\omega_n}\cap T_{\omega_n}T_{\omega_{n-1}}I_{k_1,\omega_{k-1}}\cap \cdots \cap T_{\omega_n}T_{\omega_{n-1}}\circ\cdots\circ T_{\omega_1}I_{k_1,\omega_1}
\end{split}
\end{equation*}
for a given $n$-tuple $\{k_n,\cdots,k_1\}$ is a connected interval. We will prove, by contradiction, that the function $\Psi_{\omega_n,\cdots,\omega_1}(x)$ is smooth in a neighborhood of $h_*$, actually $C^1$,  when $h_*$ is in the interior of one or several of the images $H_{\omega_1,\cdots,\omega_n}^{k_1,\cdots, k_n}$ described above. If $\Psi_{\omega_n,\cdots,\omega_1}(x)$ is not smooth in a neighborhood of $h_*$ on any interval ${\mathcal O}\supset h_*$, we can find a sequence $(\omega_n,\cdots \omega_{n-l})$, $1\le l\le n-1$ and at least an endpoint of one of the domain of injectivity of $T_{\omega_{n-l}}$, call it $c^*_{\omega_{n-l}}$,  such that $T_{\omega_{n}}\circ\cdots\circ T_{\omega_{n-l}}(c^*_{\omega_{n-l}})\in {\mathcal O}$. We now proceed by induction and we begin by showing that $\Psi_{\omega}$ is smooth in a neighborhood of $h_*$ for any choice of $\omega$. If not there will a point $c_{l,\omega}$ (see above, where we now mean that this point is one of the two boundaries of the interval $I_{l,\omega}$),  such that $T_{\omega_l} c_{l,\omega}\in (h_*-\sigma'/2,h_*+\sigma'/2)$. Let us now take the corresponding point $c_{l,0}$ of $T$ which will belongs to $C_0$. We have
$$
|T_{l,0}(c_{l,0})-T_{l,\omega}(c_{l,\omega})|\le |T_{l,0}(c_{l,0})-\bar{T}_{l,0}(c_{l,\omega})|+|\bar{T}_{l,0}(c_{l,\omega})-T_{l,\omega}(c_{l,\omega})|
$$
 By the uniform conditions (\ref{C1}) and (\ref{C2}) we can take $\eps$ smaller that a certain $\eps_{\sigma'}$ depending on $\sigma'$ such that
 $$
 |T_{l,0}(c_{l,0})-\bar{T}_{l,0}(c_{l,\omega})|\le \sup_{l,\ x\in I_{l,0}}|D\bar{T}_{l,0}||c_{l,0}-c_{l,\omega}|\le \sigma'/2
 $$
 and
 $$
 |\bar{T}_{l,0}(c_{l,\omega})-T_{l,\omega}(c_{l,\omega})|\le \sigma'/2
 $$
 and this implies
 $$
 |T_{l,0}(c_{l,0})-h_*|\le \sigma'
 $$
 which contradicts the above statement $T^n C_0\cap H_0=\emptyset.$ Fix $\sigma=\sigma'/2$;  we now continue the induction process by supposing that for any sequence $(\omega_1,\cdots,\omega_n)$ the function $\Psi_{\omega_n,\cdots,\omega_1}$ is smooth on the interval $(h_*-\sigma, h_*+\sigma)$ and we want to prove that the function $\Psi_{\omega_{n+1},\cdots,\omega_{1}}$ is still smooth for the sequence $(\omega_1,\cdots,\omega_{n+1})$. For that it will be enough to repeat the previous argument by noticing that 
 $$\Psi_{\omega_{n+1},\cdots,\omega_{1}}={\mathcal L}_{\omega_{n+1}}\Psi_{\omega_n,\cdots,\omega_1}.$$ 
 This shows that $\Psi_{\omega_n,\cdots,\omega_1}(x)$ is smooth in a neighborhood of $h_*$. We now take two points $x,y$ in the interval $(h_*-\sigma, h_*+\sigma)$; we introduce the notation $T^{-1}_{\bar{\omega}_n}:=T^{-1}_{k_1, \omega_1}\circ\cdots\circ T^{-1}_{k_n, \omega_n}$ and we compute 
 $$
|\Psi_{\omega_n,\cdots,\omega_1}(x)-\Psi_{\omega_n,\cdots,\omega_1}(y)|\le
 $$
 \begin{equation}\label{BE}
 \sum_{k_1,\ldots,k_n}\left|\frac{1}{|D(T_{\omega_n}\circ\cdots\circ T_{\omega_1})(T^{-1}_{\bar{\omega}_n}(x))|}-\frac{1}{|D(T_{\omega_n}\circ\cdots\circ T_{\omega_1})(T^{-1}_{\bar{\omega}_n}(y))|}\right|
 \end{equation}
 where we skip the characteristic function since both $x,y$ will be in the same
 $
 \Omega_{\omega_1,\cdots,\omega_n}^{k_1,\cdots, k_n}.
 $
 We have
 \begin{equation}\label{dess}
 \begin{split}
 &D\left[\frac{1}{|D(T_{\omega_n}\circ\cdots\circ T_{\omega_1})(z)}\right]\\
 &=\sum_{k=0}^{n-1}\frac{ D^2T_{\omega_{n-k}}\left(\prod_{l=1}^{n-1-k}T_{\omega_{n-l}}z\right)}{\left[DT_{\omega_{n-k}}\left(\prod_{l=1}^{n-1-k}T_{\omega_{n-l}}z\right)\right]^2\prod_{j=0}^{k}DT_{\omega_{n-j+1}}\left(\prod_{l=1}^{n-j}T_{\omega_{n-l}}z\right)}.
 \end{split}
 \end{equation}
 Since $$
 \sup_{\omega, x\in I_{k, \omega}}\left|\frac{D^2T_{\omega}(x)}{DT_{\omega}(x)}\right|\le C_1<\infty,
 $$
 which follows from our assumptions on the map $T$ and its perturbations,  and
 $$
 \inf_{\omega, x\in I_{k, \omega}}\left|DT_{\omega}(x)\right|\ge \beta >2,
 $$
 the sum in \eqref{dess} will be bounded by $C_1$ times the sum of a geometric series of common ratio $\beta^{-1}$: we call $C_2$ the upper bound thus found. Therefore, by the mean value theorem, we have
 $$
 (\ref{BE})\le C_2 \sum_{k_1,\cdots,k_n}|T^{-1}_{\bar{\omega}_n}(x)-T^{-1}_{\bar{\omega}_n}(y)|\le |x-y|\sum_{k_1,\cdots,k_n}\frac{1}{|D(T_{\omega_1}\circ\cdots\circ T_{\omega_n})(T^{-1}_{\bar{\omega}_n}(\zeta))|}
 $$
 where $\zeta\in (x,y)$. We can replace $\xi$ with $x$ (or $y$) by a standard distortion argument which works in our case by the assumptions we have on the maps $T_{\omega}$. By modifying the constant $C_2$ into a constant $C_3$ which takes into account the distortion factor we finally have
 \begin{equation}\label{integ}
 (\ref{BE})\le C_3 |x-y|\sum_{k_1,\cdots,k_n}\frac{1}{|D(T_{\omega_1}\circ\cdots\circ T_{\omega_n})(T^{-1}_{\bar{\omega}_n}(x))|}.
 \end{equation}
 We now integrate \eqref{integ} over the noise and get, for any $x,y\in I_{\tilde\sigma}$,
$$|\mathcal{L}_{\eps}^n1(x)-\mathcal{L}_{\eps}^n1(y)|\le C_3 \sigma \mathcal{L}_{\eps}^n1\le C_3 \sigma (1+B),$$
where $\mathcal L^n_{\eps}1\le 1+B$ is obtained from the Lasota-Yorke inequality in Proposition \ref{prop1}. The Lemma finally follows by choosing $K=C_3(1+ B)$.
\end{proof}

In \cite{GHW}, the authors dealt with the deterministic version of this section of our paper. In \cite{GHW}, the notion of the \textit{postcritical set} of a map was used extensively to identify the location of discontinuities of the invariant density. However, in the random setting, the notion of a postcritical set of $\mathfrak T_\eps$ does not exist. Following the ideas of our Lemma \ref{Le2} above, we develop an approach which is suitable for our random setting. We first recall a useful representation of a $1$-dimensional function of bounded variation\footnote{The usefulness of such a representation in studying stability and \textit{response theory} of $1$-dimensional invariant densities was popularized by Baladi \cite{Ba1}.}. For $f\in BV(I)$, choose a version of $f$ with regular discontinuities: for each $x$, $f(x)=(\lim_{y\to x^-}f(y)+\lim_{y\to x^+}f(y))/2$. Then, $f$ can be uniquely decomposed as $f=f^{reg}+f^{sal}$, where the regular term $f^{reg}$ is continuous and of bounded variation, with $V(f^{reg})\le V(f)$, while the singular (or saltus) part $f^{sal}$ is the sum of jumps 
$$f^{sal}=\sum_{u\in S}s_uH_u, $$
where $S$ is a countable set, and 
$$ 
H_u(x)=\left\{\begin{array}{ccc}
-1&\mbox{if $ x<u$}\\
-\frac{1}{2}&\mbox{if $x=u$}\\
0&\mbox{if $x>u$}
\end{array}
\right. .
$$
This representation imposes the boundary condition $f^{sal}(1)=0$. Moreover, the variation of the singular part satisfies
$$V(f^{sal})=\sum_{u\in S}|s_{u}|\le V(f).$$
Call 
$$F_{n,\eps}:=\frac1n \sum_{k=0}^{n-1}{\mathcal L}^k_{\eps}1:=\frac1n \sum_{k=0}^{n-1}\int_{\bar\Omega_{\eps}} {\mathcal L}^k_{\bar\omega}1 d\theta^{\infty}_{\eps}(\bar\omega),$$ 
where ${\mathcal L}^k_{\bar\omega}={\mathcal L}_{\omega_1}\cdots {\mathcal L}_{\omega_k}.$ Note that, by ${\bf(R1)}$, we have
$$\lim_{n\to\infty}||F_{n,\eps}-\rho_{\eps}||_1=0.$$
We have already shown in Lemma \ref{Le2} that the set of discontinuities of ${\mathcal L}^n_{\bar\omega}$ is given by
$$
{\mathcal S}_{n,\bar\omega}= \cup_{j=0}^{n-1}\cup_{k_{n-j}}T_{\omega_{n}}T_{\omega_{n-1}} \cdots T_{\omega_{n-j}} \partial I_{k_{n-j},\omega_{n-j}}
$$
where
$k_{n-j}$ runs over the domains of injectivity of $T_{\omega_{n-j}}$ and $\partial I$ denotes the endpoint of the connected interval $I$.
We rewrite this set as
$$
{\mathcal S}_{n,\bar\omega}=\cup_{j=0}^{n-1} \hat{{\mathcal S}}_{j,\bar\omega}, \ \ \text{ where }\ \ \hat{{\mathcal S}}_{j,\bar\omega}:=\cup_{k_{n-j}}T_{\omega_{n}}T_{\omega_{n-1}} \cdots T_{\omega_{n-j}} \partial I_{k_{n-j},\omega_{n-j}}.
$$
We also write
$$\tilde{\mathcal S}_n:=\cup_{\bar\omega}{\mathcal S}_{n,\bar\omega} \ \ \text{ and }\ \ \hat{\mathcal S}_j=\cup_{\bar\omega} \hat{{\mathcal S}}_{j,\bar\omega}.$$
Notice that $\hat S_j\subset \tilde{\mathcal S}_n$. Finally, we write
$$\mathcal S:=\cup_{n\ge 1}\tilde{\mathcal S}_n.$$
\begin{lemma}\label{finite}
\text{ }
\begin{enumerate}
\item $\tilde{\mathcal S}_n$ is a finite set and $\mathcal S$ is a countable set.
\item $[0,1]\setminus\tilde{\mathcal S}_n$ is a finite collection of open intervals.
\item For any $\bar\omega:=(\omega_n,\dots,\omega_1)$, $T_{\bar\omega}$ is $C^2$ on each open interval belonging to $[0,1]\setminus\tilde{\mathcal S}_n$.
\end{enumerate}
\end{lemma}
\begin{proof}
Since $\Omega_{\eps}$ is a finite set, (1) follows by definition of $\tilde{\mathcal S}_n$ and $\mathcal S$. (2) is a consequence of (1). (3) follows by definition of  $\tilde{\mathcal S}_n$.
\end{proof} 
\begin{remark}\label{count}
As we have mentioned in footnote 4, if one relaxes the assumption that $\Omega_{\eps}$ is finite, even to the case where $\Omega_{\eps}$ is countable but not finite, then intractable complications arise. Indeed, if $\Omega_{\eps}$ is countable but not finite, one can construct an example of a random metastable system such that in an interval, $[u,v]$, say, each rational point is in the postcritical set of a different $T_\omega$. For a system like this, the complement set $[0,1]\setminus\tilde{\mathcal S}_n$ becomes uncountable, and its intersection with $[u,v]$ will not contain any interval\footnote{One can easily construct a random system with such a bad behaviour. For example, let us write the formula of the first branch for the map $T_0$ in Figure 2. We have $T_0(x):=3x$ for $0\le x\le 1/6$. For a fixed rational number $\eps$, Let $\Omega_{\eps}=\mathbb Q\cap [0,\eps]$. Then, each $\omega\in \Omega_{\eps}$ is a rational number. Further, define the first branch of $T_{\omega}(x):= (3+6\omega)x$, for $0\le x\le 1/6$. Then  $\omega+1/2$ is in the postcritical set of $T_\omega$. In particular $\tilde S_1\cap [1/2, 1/2+\eps]=\{\omega+1/2 \}_{\omega\in\Omega_{\eps}}$. Consequently, $[0,1]\setminus(\tilde S_1\cap [1/2, 1/2+\eps])$ is an uncountable set and it does not contain any interval.}. This makes it impossible to obtain regularity properties of $\rho_{\eps}^{reg}$, the regular part of the stationary density $\rho_{\eps}$.  
\end{remark}
\begin{lemma}\label{Le3'}
The discontinuities of $F_{n,\eps}$ belong to $\tilde{\mathcal S}_n$.
\end{lemma}
\begin{proof}
We prove the lemma by contradiction. Namely, let $a$ be a discontinuity point for $F_{n,\eps}$, and suppose that $a\notin \tilde{\mathcal S}_n,$. Then $\forall \bar\omega$ and $0\le k\le n-1$, we have that 
\begin{equation}\label{g1}
\lim_{x\rightarrow a}{\mathcal L}^k_{\bar\omega}1(x)={\mathcal L}^k_{\bar\omega}1(a).
\end{equation} 
Since, for sufficiently small $\epsilon$, each ${\mathcal L}_{\omega}$, $\omega\in \Omega_{\eps}$, satisfies a Lasota-Yorke inequality with uniform constants (see Proposition \ref{prop1})
$$V\mathcal L_{\omega}f\le\eta Vf+B_0||f||_1,$$
we concatenate the transfer operators and get the existence of a positive constant $C'$, independent of the realization, for which 
$$||{\mathcal L}^k_{\bar\omega}1||_{\infty}\le ||{\mathcal L}^k_{\bar\omega}1||_{BV}\le C'.$$ 
Therefore we can use \eqref{g1} and apply the Lebesgue dominated convergence theorem to get
\begin{equation*}
\begin{split}
F_{n,\eps}(a)&=\frac1n \sum_{k=0}^{n-1}\int_{\bar\Omega_{\eps}}{\mathcal L}^k_{\bar\omega}1(a)=\frac1n \sum_{k=0}^{n-1}\int_{\bar\Omega_{\eps}}\lim_{x\rightarrow a}{\mathcal L}^k_{\bar\omega}1(x)\\
&=\lim_{x\rightarrow a}\frac1n \sum_{k=0}^{n-1}\int_{\bar\Omega_{\eps}}{\mathcal L}^k_{\bar\omega}1(x)=\lim_{x\rightarrow a}F_{n,\eps}(x).
\end{split}
\end{equation*}
Thus, a contradiction. 
\end{proof} 
Let
$$
\#u=\inf\{ j\ge 1;\ u\in  \hat{{\mathcal S}}_{j}\}.
$$
\begin{lemma}\label{Le4'}
\text{}
\begin{enumerate}
\item The discontinuity set of $\rho_{\eps}$ is subset of $\mathcal S$. If we write $\rho_{\eps}^{\text{sal}}:=\sum_{u\in \mathcal S}s_u H_u$, then
$\sum_{u\in {\mathcal S}, \#u>m}|s_u|\le \eta^{-m}C'$.\\ 
\item $\rho_{\eps}^{reg}$ is Lipschitz continuous. 
\end{enumerate}
\end{lemma}
\begin{proof}
Since $F_{n,\eps}$ is a function of bounded variation, $F_{n,\eps}$ can be decomposed as $F_{n,\eps}=F_{n,\eps}^{reg}+F_{n,\eps}^{sal}$. By Lemma \ref{Le3'}, the discontinuity set of $F_{n,\eps}$ is contained in $\tilde{\mathcal S}_n\subset \mathcal S$. Then we can write the saltus part of $F_{n,\eps}$ as
$$
F_{n,\eps}^{sal}=\sum_{u\in {\mathcal S}}s_{u,n}H_u.
$$
We first show that $\sum_{\#u>m}|s_{u,n}|$ decays exponentially fast in $m$. To show this, observe that if $m\ge n$ then $\sum_{u\in S; \ \#u>m}|s_{u,n}|=0$ and
$$\sum_{u\in {\mathcal S}; \ \#u>0}|s_{u,n}|\le V(F_{n,\eps})\le V({\mathcal L}_{\eps}^k 1)\le C'.$$
Moreover, 
$$
F_{n,\eps}=\frac{n-1}{n} \mathcal L_{\eps} F_{n-1,\eps}+\frac1n.
$$
Now  for $\#u>1$, i.e. $u\notin \hat{\mathcal S}_1$, which means that $\forall \omega$, each preimage $T^{-1}_{k,\omega}$ is continuous at $u$. Therefore the contributions to the jump $s_{u,n}$ come from the jumps of $F_{n-1,\eps}$ and are located at the points $v$ such that $T_{\omega}v=u$, precisely
$$
|s_{u,n}|\le \frac{n-1}{n} \frac{1}{\eta}\int \sum_{v\in {\mathcal S}; T_{\omega'}v=u} |s_{v,n-1}|d\theta_{\eps}(\omega').
$$
Then if $0<m<n$ we have 
\begin{equation*}
\begin{split}
\sum_{\#u>m}|s_{u,n}|&\le  \sum_{\#u>m}\frac{n-1}{n} \frac{1}{\eta}\int  \sum_{v\in {\mathcal S}; T_{\omega'}v=u} |s_{v,n-1}|d\theta_{\eps}(\omega')\\
&\le \frac{n-1}{n} \frac{1}{\eta}\sum_{\#u>m-1}|s_{u,n-1}|\le \cdots\\
&\le \frac{n-m}{n} \frac{1}{\eta^m}\sum_{\#u>0}|s_{u,n-m}|\le \frac{1}{\eta^m}C'.
\end{split}
\end{equation*}
 By using a diagonalization argument, we may find a subsequence $n_j$ such that for each $u\in {\mathcal S}$, $s_{u, n_j}$ converges as $n_j\rightarrow \infty$ to some number, which we call $s'_u$. Let us construct the function 
 $$
 F_{\eps}^{sal}:=\sum_{u\in {\mathcal S}}s'_u H_u.
 $$
It is clear that $F_{n,\eps}^{sal}$ converges in $L^1$ to $F_{\eps}^{sal}$ and that for each $m$: $\sum_{u\in {\mathcal S}, \#u>m}|s'_u|\le \eta^{-m}C'$. We now show that $F_{\eps}^{sal}$ coincides with $\rho_{\eps}^{sal}$ and $F^{reg}$, a limit point of $F_{n,\eps}^{reg}$, coincides with $\rho_{\eps}^{reg}$. Indeed, by (3) of Lemma \ref{finite}, $F_{n,\eps}$ is $C^1$ on the open intervals of $[0,1]\setminus\tilde{\mathcal S_n}$. Therefore, on $[0,1]\setminus\tilde{\mathcal S_n}$, the derivative of $F_{n,\eps}$ coincides with the derivative of $F_{n,\eps}^{reg}$ (since $F_{n,\eps}^{sal}$ is constant on each interval of $[0,1]\setminus\tilde{\mathcal S_n}$). Then, using the arguments in Lemma \ref{Le2}, in particular estimates \eqref{dess} and \eqref{integ}, we obtain a uniform estimate on the derivative of $F_{n,\eps}^{reg}$ on $[0,1]\setminus\tilde{\mathcal S_n}$. By (2) Lemma \ref{finite}, $[0,1]\setminus\tilde{\mathcal S_n}$ consists of a finite number of intervals. Thus, since by definition $F_{n,\eps}^{reg}$ is a continuous function, we conclude that $F_{n,\eps}^{reg}$ is uniformly Lipschitz on $[0,1]$. Consequently, by the Arzel\`a-Ascoli theorem, we can find a continuous function $F^{reg}$ such that some subsequence of $\{F_{n_j,\eps}^{reg}\}$ converges in $L^{\infty}$, and hence in $L^1$, to $F^{reg}$. By the uniqueness of the decomposition $\rho_{\eps}=\rho^{sal}_{\eps}+\rho^{reg}_{\eps}$, we get that $F^{reg}=\rho_{\eps}^{reg}$ and $F^{sal}=\rho_{\eps}^{sal}$.
\end{proof}
\begin{lemma}\label{Le5'}\text{ }
\begin{enumerate}
\item For any $\tau>0$ there exist an $n$, $0\le n<\infty$, a $\sigma$ depending on $n$, and an $\eps_{\sigma}$, depending on $\sigma$, such that   $\forall \eps\le \eps_{\sigma}$ and $\tilde{\sigma}\le \sigma$, we have
$$V_{I_{\tilde\sigma}}\rho_{\eps}^{sal}<\eta^{-n} C'<\tau,$$ 
where $I_{\tilde\sigma}=[h_* -\tilde{\sigma}, h_*+\tilde{\sigma}]$, $*\in\{l,r\}$ and $C'$ is the constant which appears in Lemma \ref{Le4'}.
\item For any $x,y\in I_{\tilde\sigma}$, there exists a uniform constant $K$ such that
$$|\rho_{\eps}^{reg}(x)- \rho_{\eps}^{reg}(y)|\le K\cdot|x-y|.$$ 
\end{enumerate}
\end{lemma}
\begin{proof}
Choose $n$ large enough so that $\eta^{-n} C'<\tau$. Then, using (1) of Lemma \ref{Le2}, for $n>0$ such that $\eta^{-n} C'<\tau$ we can find a $\sigma$, depending on $n$, and an $\eps_{\sigma}$, depending on $\sigma$, such that $\forall \eps\le \eps_{\sigma}$ and $\tilde{\sigma}\le \sigma$, $\mathcal L^n_{\eps} 1$ is a $C^1$-function on $I_{\tilde\sigma}$. Consequently, using Lemma \ref{Le4'} with $m=n$, we obtain part (1) of the lemma. Part (2) of the lemma is a consequence of part (2) of Lemma \ref{Le4'}.
\end{proof} 

We now observe that by the Lasota-Yorke inequality of Proposition \ref{prop1} and the compactness of the unit closed ball of $||\cdot||_{BV}$ in $L^1$, there are subsequences of values $\tilde\eps\in \Omega_{\eps}$ for which $\rho_{\tilde \eps}$  converges in the $L^1$ norm to the density of an acim on $I$. Obviously, such a density will surely be a convex combination of $\rho_l$ and $\rho_r$. We would like to  show that those accumulation points  are always  the same so that $\rho_{\eps}$ will admit a limit when $\eps\rightarrow 0$. The first step is to characterize the coefficients of the convex combination in terms of the behavior of the density in the neighborhoods of the holes. We first have:
\begin{lemma}\label{Le6}
Let $\rho_0$ be the accumulation point of the convergent subsequence   $\rho_{\tilde\eps}$ with  $\tilde\eps\in \Omega_{\eps}$; then there exists $0\le \alpha\le 1$ such that:\footnote{Later in the proof of (1) of Theorem \ref{main}, we will find $\frac{\alpha}{1-\alpha}$ explicitly. In particular, we will show that $l.a.h.r=\frac{\alpha}{1-\alpha}$.}
\begin{enumerate}
\item $\rho_0=\alpha\rho_l+(1-\alpha)\rho_r$;
\item $\underset{\tilde\eps\to 0}{\lim}\underset{\omega\in\Omega_{\tilde\eps}}{\sup} \underset{x\in H_{l,\omega}}{\sup}|\rho_{\tilde \eps}(x)-\alpha\rho_l(x)|=0;$
\item $\underset{\tilde\eps\to 0}{\lim}\underset{\omega\in\Omega_{\tilde\eps}}{\sup} \underset{x\in H_{r,\omega}}{\sup}|\rho_{\tilde \eps}(x)-(1-\alpha)\rho_r(x)|=0.$
\end{enumerate}
\end{lemma}
\begin{proof}
We argued above that (1) is true.  The proof of (3) will be identical to that of (2). Thus, we only prove (2). Observe that for all $\omega\in\Omega_{\tilde\eps}$ we have $H_{l,\omega}\subseteq H_{l,\tilde\eps}$. Therefore,
\begin{equation*}
\underset{\omega\in\Omega_{\tilde\eps}}{\sup} \underset{x\in H_{l,\omega}}{\sup}|\rho_{\tilde\eps}(x)-\alpha\rho_l(x)|\le \underset{x\in H_{l,\tilde\eps}}{\sup}|\rho_{\tilde\eps}(x)-\alpha\rho_l(x)|.
\end{equation*}
Thus, it is enough to prove that
\begin{equation}\label{3}
\underset{\tilde\eps\to 0}{\lim} \underset{x\in H_{l,\tilde\eps}}{\sup}|\rho_{\tilde\eps}(x)-\alpha\rho_l(x)|=0.
\end{equation}
We will prove (\ref{3}) by contradiction. For this purpose we will suppose that the there exists a $C>0$ and a subsequence $\eps' \to 0$ of $\tilde\eps$ such that for each $\eps'$ there is an $x_{\eps'}\in H_{l,\eps'}$, $x_{\eps'}\to h_l$ with $|\rho_{\eps'}(x_{\eps'})-\alpha\rho_l(x_{\eps'})|>C$.  By Lemma 1 of \cite{GHW}, for sufficiently small $\bar{\sigma}>0$, for any $x\in [h_l-\bar\sigma,h_l+\bar\sigma]:=I_{\bar\sigma}$ and $x_{\eps'}\in I_{\bar\sigma}$ we have
\begin{equation}\label{4}
|\alpha\rho_l(x_{\eps'})-\alpha\rho_l(x)|<2C/5.
\end{equation}
Choose $n$ big enough so that $C'\eta^{-n}<C/5$. By Lemma \ref{Le5'} we can find a $\sigma$ depending on $n$, and an $\eps'_{\sigma}$, depending on $\sigma$, such that $\forall \eps'\le \eps'_{\sigma}$ and $\tilde{\sigma}\le \sigma$, we have
$$V_{I_{\tilde\sigma}}\rho_{\eps'}^{sal}<C/5,$$ 
and for any $x,y\in I_{\tilde\sigma}$, there exists a uniform constant $K$ such that
$$|\rho_{\eps}^{reg}(x)- \rho_{\eps}^{reg}(y)|\le K\cdot|x-y|\le K\tilde\sigma.$$  
We make sure to take $\tilde\sigma$ small enough such that $K\tilde\sigma<C/5$, and $\tilde\sigma\le \bar\sigma$. Then for $x, x_{\eps'}\in I_{\tilde\sigma}$ we have
\begin{equation}\label{5}
|\rho_{\eps'}(x_{\eps'})-\rho_{\eps'}(x)|\le |\rho_{\eps'}^{sal}(x_{\eps'})-\rho_{\eps'}^{sal}(x)|+|\rho_{\eps'}^{reg}(x_{\eps'})-\rho_{\eps'}^{reg}(x)|< 2C/5
\end{equation}
Thus, for any $x\in I_{\tilde\sigma}$ and $x_{\eps'}\in I_{\tilde\sigma}$, by (\ref{4}), (\ref{5}) and the assumption that $|\rho_{\eps'}(x_{\eps'})-\alpha\rho_l(x_{\eps'})|>C$, we obtain
\begin{equation}\label{6}
\begin{split}
|\rho_{\eps'}(x)-\alpha\rho_l(x)|&\ge|\rho_{\eps'}(x_{\eps'})-\alpha\rho_l(x_{\eps'})|-|\alpha\rho_l(x) -\alpha\rho_l(x_{\eps'})|-|\rho_{\eps'}(x_{\eps'})-\rho_{\eps'}(x)|\\
&\ge C-2C/5-2C/5= C/5.
\end{split}
\end{equation}
The estimate in (\ref{6}) contradicts the fact that $\underset{\tilde\eps\to 0}{\lim}||\rho_{\tilde\eps}-\rho_0||_1=0.$ Therefore, (\ref{3}) holds and the lemma follows.
\end{proof}
We are now ready to prove Theorem \ref{main}.
\begin{proof} (of Theorem \ref{main})
Let $\rho_0=\alpha\rho_l+(1-\alpha)\rho_r$ be the limit of $\rho_{\tilde\eps}$ obtained in (1) of Lemma \ref{Le6}. We will show that $\alpha/(1-\alpha)=l.a.h.r.$ This will imply the first part of Theorem \ref{main}.
We first observe that
\begin{equation}\label{7}
\begin{split}
&\int_{\Omega_{\tilde\eps}}\mu_{\tilde\eps}(H_{l,\omega})d\theta_{\tilde\eps}(\omega)=\int_{\Omega_{\tilde\eps}}\int_{H_{l,\omega}}\rho_{\tilde\eps}dxd\theta_{\tilde\eps}(\omega)\\
&= \int_{\Omega_{\tilde\eps}}\int_{H_{l,\omega}}\alpha\rho_{l}dxd\theta_{\tilde\eps}(\omega)+\int_{\Omega_{\tilde\eps}}\int_{H_{l,\omega}}(\rho_{\tilde\eps}-\alpha\rho_l)dxd\theta_{\tilde\eps}(\omega)\\
&=\alpha\int_{\Omega_{\tilde\eps}}\mu_{l}(H_{l,\omega})d\theta_{\tilde\eps}(\omega)+O\left(\sup_{\omega\in\Omega_{\tilde\eps}}\sup_{x\in H_{l,\omega}}|\rho_{\tilde \eps}(x)-\alpha\rho_l(x)|\right)
\int_{\Omega_{\tilde\eps}}m(H_{l,\omega})d\theta_{\tilde\eps}(\omega) .
\end{split}
\end{equation}
Let $h_l$ be the infinitesimal hole in $I_l$. By condition {\bf (A4)} $\rho_0$ is continuous on $h_l$. This implies $\lim_{\tilde\eps\to 0}\frac{\mu_{l}(H_l,\omega)}{m(H_l,\omega)}=\rho_l(h_l)>0$ by condition {\bf (A5)}. Therefore, by using (2) Lemma \ref{Le6} and (\ref{7}), we obtain
\begin{equation}\label{8}
\int_{\Omega_{\tilde\eps}}\mu_{\tilde\eps}(H_{l,\omega})d\theta_{\tilde\eps}(\omega)=\alpha\int_{\Omega_{\tilde\eps}}\mu_{l}(H_{l,\omega})d\theta_{\tilde\eps}(\omega)+\small{o}(1)\cdot\int_{\Omega_{\tilde\eps}}\mu_{l}(H_{l,\omega})d\theta_{\tilde\eps}(\omega).
\end{equation}
Similarly we can obtain
\begin{equation}\label{9}
\int_{\Omega_{\tilde\eps}}\mu_{\tilde\eps}(H_{r,\omega})d\theta_{\tilde\eps}(\omega)=(1-\alpha)\int_{\Omega_{\tilde\eps}}\mu_{r}(H_{r,\omega})d\theta_{\tilde\eps}(\omega)+\small{o}(1)\cdot\int_{\Omega_{\tilde\eps}}\mu_{r}(H_{r,\omega})d\theta_{\tilde\eps}(\omega).
\end{equation}
Using Lemma \ref{Le1} together with equations (\ref{8}) and (\ref{9}) complete the proof of Theorem \ref{main}.
\end{proof}  
\section{discussion}
In this paper we studied both open and metastable systems. In Theorem \ref{main1}, we rigorously derived first-order asymptotic escape rate formulae for piecewise expanding maps of the interval with \textit{random holes}. In \cite{KL2}  the authors studied a \textit{deterministic} sequence of \textit{nested} holes and derived a first-order asymptotic escape rate formulae for piecewise expanding maps of the interval with holes. In this paper the sequence of holes is \textit{random}. Moreover, the random holes we consider are \textit{not necessarily} nested. Consequently our work generalizes the pioneering work of \cite{KL2} to the random setting. It would be interesting to explore whether one can obtain higher order approximations in the random setting, even for a simpler class of maps, such as circle maps. A study for higher order approximations of escape rates for deterministic holes using circle maps was carried in \cite{Det1}. Other possible generalizations, such as higher dimensional systems, were discussed earlier in the introduction. 

In the second part of the paper we studied the problem of approximating the unique stationary density for a random system which initially contains two separate ergodic components. We have shown in Theorem \ref{main} that the unique stationary density of the random system can be approximated by a weighted combination of the two initially separate ergodic components of the deterministic system. Furthermore, we have shown that the weights are given by the ratio of the escape rates on the individual ergodic components. A deterministic perturbation version of the above result was first studied in \cite{GHW} for piecewise expanding maps. A version of \cite{GHW} for intermittent maps was obtained in \cite{BV}. It is worth restating (see introduction) that it would be also interesting to obtain a random version in the case of intermittent maps.\\

{\bf Acknowledgment} We would like to thank anonymous referees for their suggestions which improved the presentation of the paper. 
\bibliographystyle{amsplain}

\end{document}